\documentclass{amsart}
\usepackage{mathscinet,amssymb,latexsym,stmaryrd,mathtools}
\usepackage{enumerate,mathrsfs,hyperref,cmll}
\usepackage{mathscinet}
\usepackage{xcolor}
\usepackage{tikz}

\theoremstyle{definition}
\newtheorem{theorem}{Theorem}[section]
\newtheorem{definition}[theorem]{Definition}

\newtheorem{claim}[theorem]{Claim}

\newtheorem{coro}[theorem]{Corollary}
\newtheorem{proposition}[theorem]{Proposition}

\newtheorem{question}[theorem]{Question}

\newtheorem{lemma}[theorem]{Lemma}

\theoremstyle{remark}

\numberwithin{equation}{section}

\newcommand{\CH}{\mathsf{CH}}

\newcommand{\nwd}{\mathsf{nwd}}
\newcommand{\hc}{\mathsf{HC}}
\newcommand{\fin}{\mathsf{fin}}
\newcommand{\add}{\mathsf{add}}
\newcommand{\non}{\mathsf{non}}
\newcommand{\cov}{\mathsf{cov}}
\newcommand{\p}{\mathbb{P}}

\newcommand{\lf}{\mathcal{L}_f}
\newcommand{\li}{\mathbb{L}(\mathcal{I}^+)}
\newcommand{\lav}{\mathbb{L}}
\newcommand{\mil}{\mathbb{MI}}
\newcommand{\mat}{\mathbb{M}}
\newcommand{\id}{\mathcal{I}}
\newcommand{\ed}{\mathcal{E}\mathcal{D}}

\newcommand{\restr}[2]{#1\upharpoonright {#2}}
\newenvironment{proofclaim}[1][Proof of the claim]{\textbf{#1.} }{\hfill \rule{0.5em}{0.5em}}

\title[Properties of Laver Forcing Associated with a Co-ideal]{Properties of Laver Forcing Associated with a Co-ideal, expressed via the Kat\v{e}tov Order}

\author[Guzmán]{Osvaldo Guzmán}
	\address{Centro de Ciencias Matem\'aticas\\ Universidad Nacional Aut\'onoma de M\'exico\\ Campus Morelia\\Morelia, Michoac\'an\\ M\'exico 58089}
	\curraddr{}
	\email{oguzman@matmor.unam.mx}
	\thanks{The first author was supported by the PAPIIT grant IA 104124}

\author[Nieto-de la Rosa]{Francisco Santiago Nieto-de la Rosa}
	\address{Posgrado Conjunto en Ciencias Matem\'aticas UNAM-UMSNH\\Morelia, Michoac\'an\\ M\'exico 58089}
	\curraddr{}
	\email{fsnieto@matmor.unam.mx, francisco.s.nieto@ciencias.unam.mx}
	\thanks{The second author has been supported by CONACYT Scholarship 1043249.}

\author[Ramos]{Ulises Ariet Ramos-García}
	\address{Centro de Ciencias Matem\'aticas\\ Universidad Nacional Aut\'onoma de M\'exico\\ Campus Morelia\\Morelia, Michoac\'an\\ M\'exico 58089}
	\curraddr{}
	\email{ariet@matmor.unam.mx}
	\thanks{The third author was supported by the PAPIIT grant IN116225}

\begin{document}
\begin{abstract}
    We study variants of classical Laver forcing defined from co-ideals and analyze their combinatorial properties in terms of the Kat\v{e}tov order. In particular, we give a Kat\v{e}tov-theoretic characterization of when Laver forcing associated with a co-ideal adds Cohen reals, and we show that such forcings never add random reals. Improving a result of B{\l}aszczyk and Shelah, we prove that the addition of Cohen reals and the Laver property are not equivalent, even in the case of ultrafilters. As an application, we investigate the problem of adding half Cohen reals without adding Cohen reals via tree-like forcings arising from co-ideals. We obtain several partial results and structural obstructions. Finally, we resolve several open questions from the literature concerning the one-to-one or constant property and cardinal invariants associated with ideals.
\end{abstract}

\maketitle

\section{Introduction.}

The Kat\v{e}tov order is a central tool in analyzing definable ideals and their associated forcing notions (See \cite{katvetov}). Given two ideals \( \mathcal{I}, \mathcal{J} \) on \( \omega \), the \emph{Kat\v{e}tov order} (denoted by \( \mathcal{I} \leq_K \mathcal{J} \)) provides a way to compare their combinatorial complexity. We say that \( \mathcal{I} \leq_K \mathcal{J} \) if there exists a function \( f: \omega \to \omega \) such that for every \( A \in \mathcal{I} \), the preimage \( f^{-1}[A] \in \mathcal{J} \). This notion reflects whether the structure of \( \mathcal{I} \) can be embedded, in a sense, into \( \mathcal{J} \). We say that two ideals $\id$ and $\mathcal{J}$ are Kat\v{e}tov equivalent (denoted by $\equiv_K$) if $\id \leq_K \mathcal{J}$ and $\mathcal{J}\leq_K\id$. A related concept is that of a \emph{restriction} of an ideal. Given an ideal \( \mathcal{I} \) and a subset \( X \subseteq \omega \), the restriction of \( \mathcal{I} \) to \( X \), denoted by \( \mathcal{I} \restriction X \), is the collection \( \{ A \subseteq X : A \in \mathcal{I} \} \). 

Some classical ideals that are used in this work are the following and are defined in \cite{hrusak}. First we have $\nwd$ that is the ideal generated by the nowhere dense trees on $2^{<\omega}$ (a tree $T$ is nowhere dense if for every $\tau\in T$ there exists $\sigma\in 2^{<\omega}$ such that no node in $T$ contains $\sigma$). It is equivalent to the ideal of nowhere dense subsets of $\mathbb{R}$ restricted to $\mathbb{Q}$. By $\fin \times \fin$ we denote the ideal on $\omega \times \omega$ consisting of those sets $A$ such that for all but finitely many $n$ the set $\{m : (n,m) \in A\}$ is finite. Finally, $\ed$ is the ideal on $\omega \times \omega$ consisting of those sets $A$ for which there exists $N$ such that for all $n > N$, the set $\{m : (n,m) \in A\}$ has at most $N$ elements.

Given $\alpha$ an ordinal and $X$ a set, we said that $T\subseteq X^{<\alpha}$ is a tree if for all $\tau\in T$, $\{\tau\upharpoonright\beta:\beta<dom(\tau)\}\subseteq T$. For $\sigma,\tau\in T$, we put $\tau\leq \sigma$ if $\sigma\subseteq \tau$. For $\tau\in T$, we define $succ_T(\tau)=\{ n \in \omega : \tau^\frown n \in T \}$. we said that $\tau$ is a splitting node if $|succ_T(\tau)|>1$. The stem of $T$ is the shortest splitting node of $T$ and is denoted by $stem(T)$. If $\tau\in T$ the subtree $\{\sigma\in T :\sigma \subseteq \tau \, \vee \tau\subseteq \sigma\}$ is denoted by $T_\tau$.

Laver, Miller, and Mathias forcing all play central roles in the study of combinatorics on the natural numbers, especially in connection with ultrafilters, cardinal invariants, and the structure of the real line. These three proper forcing notions are fundamental tools in controlling the combinatorial properties of subsets of  $\omega$.

Laver forcing (denoted by $\mathbb{L}$), introduced by Richard Laver in the 1970s, was designed to add a real with strong combinatorial properties, namely a \emph{Laver real}. Formally, a condition in Laver forcing is a tree $T \subseteq \omega^{<\omega}$ such that for all $\tau\leq stem(T)$, $succ_T(\tau)$ is infinite. The ordering is by inclusion: $S \leq T$ if $S \subseteq T$. Also, we put $T\leq_0 S$ if $T\leq S$ and $stem(T)=stem(S)$. 

Laver forcing adds a dominating real, but it does not add Cohen reals, in fact it has the Laver property. It is often used in the study of combinatorial cardinal invariants and the preservation of certain large cardinal properties.

Mathias forcing (denoted by $\mathbb{M}$) is a related but distinct forcing notion defined to study ultrafilters on $\omega$. A condition in Mathias forcing is a pair $(s,A)$ where $s \in [\omega]^{<\omega}$ is finite and $A \subseteq \omega$ is infinite, with $\max(s) < \min(A)$. The extension relation is defined by $(t,B) \leq (s,A)$ if $s \subseteq t \subseteq s \cup A$ and $B \subseteq A \setminus (t \setminus s)$. Mathias forcing adds a dominated real and is closely related to Ramsey theory and the structure of ultrafilters, and also has the Laver property.

Miller forcing (denoted by $\mathbb{MI}$) is another classical forcing notion used to add a real with specific combinatorial features. A condition in Miller forcing is a tree $P \subseteq \omega^{<\omega}$ such that:
\begin{itemize}
    \item $P$ is a subtree of $\omega^{<\omega}$;
    \item for every $\tau \in P$, either $\tau$ has exactly one immediate successor in $P$, or it has infinitely many;
    \item there are infinitely many nodes $\tau \in P$ such that $succ_P(\tau)$ is infinitely (i.e., the splitting nodes are dense in the tree).
\end{itemize}
The ordering is again by inclusion: $P' \leq P$ if $P' \subseteq P$.

Miller forcing adds a so-called \emph{Miller real}, which is unbounded in the Baire space $\omega^\omega$, but not dominating. It is a powerful tool in the study of small sets of reals, covering properties and the Kat\v{e}tov order of ideals. Unlike Cohen or random forcing, Miller forcing is proper and does not add random or Cohen reals; even more, it has the Laver property. Miller forcing preserves many delicate combinatorial structures and is often used in consistency results concerning the interaction of cardinal characteristics and definable families of reals.

All three of these forcing notions have variants defined with co-ideals (or, more classically, with filters and ultrafilters). If $\mathcal{I}$ is an ideal on $\omega$, the conditions in $\lav$ and $\mil$ modify the definition of conditions by requiring that the set of successors belongs to $\id^+$ (the complement of $\id$), rather than being simply infinite. For $\mat$, we request that the set $A$ belongs to $\id^+$. The notation $\li$, $\mil(\id^+)$ and $\mat(\id^+)$ refer respectively to the Laver, Miller, and Mathias forcings associated with the co-ideal $\id$. These forcings have been studied in \cite{prikry}, \cite{4more},\cite{blaszczyk}. The particular case of Laver forcing associated with a co-ideal has also been studied from other perspectives. For example, in \cite{khomskii} the analysis takes a topological approach.

Another relevant variant of Miller forcing for this work is the full-splitting Miller forcing. A Miller tree $T$ is a full-splitting Miller tree if whenever $succ_T(\tau)$ is infinite, it is actually equal to $\omega$. We denote by $\mathbb{FM}$ the forcing consisting of all full-splitting Miller trees ordered by inclusion.

In general, a countable support iteration of forcings that do not add Cohen or random reals may add a Cohen or random real. The Laver property defined in \ref{LP} is important for two main reasons: A forcing with the Laver property does not add neither Cohen nor random reals, and countable support iteration of forcings with the Laver property has the Laver property.

Keeping this in mind, a natural question is when $\li$, $\mil(\id^+)$ and $\mat(\id^+)$ add Cohen reals, random reals or they have the Laver property. In \cite{4more} (Theorem 9.10) the authors showed the following theorem.

\begin{theorem} [Farah and Zapletal \cite{4more}]\label{FZ}
    Let $\id$ be an ideal on $\omega$. The following are equivalent:
    \begin{enumerate}
        \item $\mat(\id^+)$ does not add Cohen reals;
        \item $\mat(\id^+)$ has the Laver property;
        \item $\id^+$ is semiselective (see Definition 2.1 in \cite{semiselective}).
    \end{enumerate}
\end{theorem}

With respect to $\mil(\id^+)$, the next result is  proved in \cite{sabok} (Theorem 3.6):

\begin{theorem}[Sabok and Zapletal \cite{sabok}]\label{SZ}
    Forcing $\mil(\id^+)$ adds Cohen reals if and only if $\nwd \leq_K \id\upharpoonright X$ 
for some $X\in\id^+$.
\end{theorem}

While for Miller and Mathias we have already characterized the ideals whose associated forcing add Cohen reals, for $\li$, in \cite{blaszczyk} the authors proved the next result just for ultrafilters. Remember that $\mathcal{F}$ is a nowhere dense filter if the dual ideal $\mathcal{F}^*$ is not Kat\v{e}tov above $\nwd$.

\begin{theorem}[B{\l}aszczyk and Shelah \cite{blaszczyk}]\label{BS}
    Let $\mathcal{U}$ be an ultrafilter. The following are equivalent:
    \begin{enumerate}
        \item $\lav(\mathcal{U})$ does not add Cohen reals;
        \item $\mathcal{U}$ is nowhere dense.
    \end{enumerate}
\end{theorem}

\textbf{Disclaimer:} In this work, the sentence “$\mathbb{P}$ adds a Cohen real” (and analogously for random reals or any other kind of reals) means that there is some $p\in\mathbb{P}$ such that 
\[
p\Vdash ``\text{there is a Cohen real over }V".
\]
Therefore, when we write “$\mathbb{P}$ does not add a Cohen real”, we mean that in any generic extension of $V$ obtained via $\mathbb{P}$, there is no Cohen real over $V$.

With this definition in mind, we proceed to describe the structure of the paper.

In Section 2 we get a theorem as Theorem \ref{SZ} that generalizes Theorem \ref{BS} and answers Question 4.18 of \cite{poke}.

In \cite{fernandez}, the author showed the following theorem for some specific ultrafilters on the finite subsets of $\omega$ related to the Laver property. All definitions used in this theorem can be found in \cite{fernandez} (Theorem 3.6). 

\begin{theorem}[Fernández-Bretón \cite{fernandez}]\label{breton}
    If $\mathcal{U}$ is a stable ordered union ultrafilter, then $\lav(\mathcal{U})$ has the Laver property.
\end{theorem}

Motivated by the four theorems mentioned above and by the relevance of the Laver property mentioned above, in Section 3 we define a family of ideals which provides a necessary and sufficient condition on an ideal $\id$ for $\li$ has the Laver property, analogous to Theorem \ref{SZ}. In particular, we prove that for an ideal $\id$ is not equivalent for $\li$ having the Laver property and $\li$ does not add Cohen reals. In addition, it is not equivalent even for ultrafilters consistently. There is a $P$-point $\mathcal{U}$ such that $\lav(\mathcal{U})$ does not add Cohen reals and it does not have the Laver property.

In Section 4 we provide necessary and sufficient conditions to determinate when $\li$ preserves outer measure in terms of Kat\v{e}tov order. This is inspired by a similar result for Miller forcing proved in \cite{sabok}.

In Section 5, we explore Question 3.2 of \cite{hc}, which concerns forcing notions that add a half Cohen real but no Cohen reals. In light of Theorem \ref{Cohen reals}, we attempt to use forcings of the form $\li$ or $\mil(\id^+)$ to address this question. Although we generalize some of the partial results presented in \cite{fullmil}, the problem remains unsolved.

Sections 6 and 7 focus on answering some questions that appear in \cite{poke}. In particular, Question 2.13 wonders if $\lav((\fin\times \fin)^+)$ has the 1-1 or constant property (defined in Section 5) studied in \cite{sabok} and \cite{poke}. In Section 6 we study cardinal invariants related to some particular ideals. In \cite{hernandez}, the authors introduce the following cardinal invariants associated with an ideal.

\begin{definition}
    Given an ideal $\id$,
    \[
    \add^*(\id)=\min\{ |\mathcal{A}| : \mathcal{A}\subseteq \id \ \text{and}\ \forall X\in \id \ \exists A\in \mathcal{A} \ (|A\setminus X|=\omega)\}
    \]

    \[
    \non^*(\id)= \min\{|\mathcal{A}| : \mathcal A \subseteq [\omega]^\omega \ \text{and}\ \forall X\in \id \ \exists A\in \mathcal{A} \ (A\cap X =^*\emptyset) \}
    \]
\end{definition}

Recall that $X=^*Y$ means the symmetric difference is finite.

These cardinal characteristics, as well as several related invariants, were further studied in \cite{hrusak}. Later, in \cite{poke}, the authors introduced an ``$\omega$-version'' of these two invariants. These variants will be defined in Section~5, where we answer Question~4.10 and partially answer Question~4.11 from that paper.

Given $S,T$ subtrees of $\omega^{<\omega}$, we write $S\leq_0 T$ if $S$ is a subtree of $T$ and $stem(S)=stem(T)$. In addition, $[T]$ is the set of branches of $T$, that is, $\{f\in \omega^{\omega}: \forall n<\omega\, (\restr{f}{n}\in T)\}$ (this can be defined for trees in $2^\omega$ similarly). 

\begin{proposition}[Folklore]
    For any ideal $\mathcal{I}$, $\li$ satisfies:
    \begin{enumerate}
        \item (pure decision property) If $T\in\li$, $\dot x$ is a $\li$-name and $T\Vdash`` \dot x\in 2"$, there is $S\leq_0 T$ such that $S\Vdash`` \dot x=0"$ or $S\Vdash`` \dot x=1"$.
        \item (continuous reading of names)
        If $T\in\li$ and $\dot x$ is a $\li$-name such that $T\Vdash`` \dot x \in\omega^\omega$", then there is $S\leq T$ and $F:[S]\rightarrow2^\omega$ a continuous ground model function such that \begin{displaymath}
        S \Vdash ``F(\dot r_{gen})= \dot x"
    \end{displaymath}
    where $\dot r_{gen}$ is the name of the generic real.
    \end{enumerate}
\end{proposition}

Let $T$ be a tree and $\tau\in T$. We write $T_\tau$ to denote the biggest subtree of $T$ with stem $\tau$, i.e., $\{\sigma \in T: \sigma\subseteq \tau \vee \tau\subseteq \sigma \}$.

\section{When can $\li$ add Cohen or random reals?}

In this section, we study when the Laver forcing associated with a co-ideal $\mathcal{I}^+$ adds Cohen or random reals. While for other classical forcing notions like Miller or Mathias, the addition of Cohen reals has been characterized in terms of the Kat\v{e}tov order (see Theorems~\ref{FZ} and~\ref{SZ}), the case of Laver forcing has remained less explored. 

Our main result in this section is a characterization, in terms of the Kat\v{e}tov order, of when $\mathbb{L}(\mathcal{I}^+)$ adds Cohen reals. In contrast, we also show that $\mathbb{L}(\mathcal{I}^+)$ never adds random reals, regardless of the ideal $\mathcal{I}$.

The following theorem is based on unpublished notes of Jörg Brendle, where the result is proved for filters instead of co-ideals.

\begin{theorem} \label{Cohen reals}
    For every ideal $\id$ on $\omega$. Then $\li$ adds Cohen reals if and only if there is $X\in\id^+$ such that $\nwd \leq_K \restr{\id}{X}$.
\end{theorem}
\begin{proof}
$(\Leftarrow)$ To prove that $\li$ adds a Cohen real, fix $X \in \id^+$ and a witness $f: X \to 2^\omega$ for $\nwd \leq_K \id \upharpoonright X$. Set $P = X^{<\omega}$. Recursively we ca construct a $\li$-name $\dot c$ such that for every $T \leq P$, if $k$ is the maximum such that $T$ decides $\dot c \upharpoonright k$, then for all $n\in succ_T(stem(T))$,

\begin{displaymath}
    T_{stem(T)^\frown n}\Vdash `` \forall i<n \, ( \dot c(k+i)=f(n)(i))".
\end{displaymath}
    
We claim that $\dot c$ is a Cohen real.

Let $N \in \nwd$, and let $R_N$ be the nowhere dense tree such that $N = [R_N]$. For any $T \leq P$, we show there exists $T' \leq T$ with $T' \Vdash ``\dot c \notin N"$. Let $k < \omega$ be the largest number such that $T$ decides $\dot c \upharpoonright k$, and let $s \in 2^k$ with $T \Vdash ``\dot c \upharpoonright k = s"$.

For each $n \in succ_T(stem(T))$, define $g_n : \omega \to 2$ such that for all $i<k$, $g_n(i)=s(i)$, and for $i \geq k$, $g_n(i)=f(n)(i-k)$. Notice that $g_n$ is a translation of $f(n)$, since $f[succ_T(stem(T))]\in \nwd^+$, then $G=\{g_n:n\in succ_T(stem(T))\}\in \nwd^+$. Since $N\in \nwd$ and $int(\bar{G})\not=\emptyset$ there is $s'\supseteq s$ such that $\langle s'\rangle\subseteq \bar G$ and $s'\notin R_N$. In particular $\langle s'\rangle\cap G\not=\emptyset$. In addition, if \[Y=\{n\in succ_T(stem(T)):s'\subseteq g_n\},\] then $Y$ is infinite.

Fix $n\in Y$ larger than $k$ and $|s'|$. By construction
\begin{displaymath}
    T_{stem(T)^\frown n}\Vdash``\forall i<n (\dot c(k+i)=f(n)(i))".
\end{displaymath}
Choosing $i=|s'|$ we have
\begin{displaymath}
    T_{stem(T)^\frown n}\Vdash``\dot c(k+i)=f(n)(i)=g_n(i+k)".
\end{displaymath}

Since $s'\notin R_N$ we have that $T_{stem(T)^\frown n}$ forces $\dot c\upharpoonright k+i+1\not\in R_N$ and we are done.

$(\Rightarrow)$ By contrapositive. Let $\dot h$ be a $\li$- name for a real and $T\in \li$. We must show that there is $S\leq T$ and $A$ a nowhere dense tree such that $S\Vdash``\dot h\in [A]"$.


For all $R\in \li$ enumerate $R\setminus \{\sigma\in R:\sigma\leq stem(R)\}$ as $\{\sigma_k^R:k\in \omega\}$ such that:

\begin{enumerate}
    \item[(a)] $\sigma_k^R<\sigma_\ell^R$ implies $k<\ell$ and
    \item[(b)] if $\sigma$ is such that $\sigma_k^R=\sigma^\frown n$, $\sigma_\ell^R=\sigma^\frown m$ and $n<m$, then $k<\ell$.
\end{enumerate}

In order to obtain the condition $S$ and the nowhere dense set $A$, we construct by recursion for each $k<\omega$ the following:

\begin{enumerate}[(i)]
    \item $\sigma_k^S\in\omega^{<\omega}$.
    \item $S^k\in \li$.
    \item $X_{k}\in \id^+$.
    \item $\forall n\in X_k, \{T^{\ell}(k,n):\ell<\omega\} \subseteq \li$.
    \item $f_k:X_{k}\rightarrow2^\omega$.
    \item $B_{k}$ a nowhere dense tree.
    \item $g(k),n_k, \ell_k\in\omega$ .
\end{enumerate}

such that $\{\sigma^S_{k}:k<\omega\}$ has the properties (a) and (b) and for all $k$:

\begin{enumerate}
    \item $\ell_k\leq k$ and $\ell_k=k$ only if $k=0$.
    \item $\sigma^S_0=stem(T)$ and $\sigma^S_{k+1}=\sigma^{S^k}_{k+1}={\sigma^{S^k}_{\ell_k}}^\frown n_k$.
    \item $S^0=\bigcup\{T^{0}(0,n):n\in X_0\}$, $S^{k+1}\leq S^{k}$ and  \\$S^{k+1}=\left(S^k\setminus S^k_{\sigma^S_{k+1}}\right)\cup \left\{T^{g(k)}(\ell_j,n_j):j\in X_{k+1}\right\}$.
    \item $X_0\subseteq succ_T(\sigma_0^S)$ and $X_{k+1}\subseteq succ_{T^{g(k)}(\ell_k,n_k)}(\sigma^S_{k+1})$.
    \item For all $\ell<\omega$ and $n\in X_k$, $T^{\ell}(k,n)$ decides $\dot h \upharpoonright \ell+1$ and
    \\$T^{\ell+1}(k,n)\leq_0 T^\ell(k,n)\leq_0 S^k_{{\sigma_k^S}^\frown n}$ 
    \item $f_k(n)(\ell)=i$ if and only if $T^\ell(k,n)\Vdash``\dot h(\ell)=i"$.
    \item $ f_k[X_k]\subseteq [B_k]$.
    \item $g(k+1)>g(k)$.
    \item For all $\bar h\in f_{k+1}[X_{k+1}]$, $\bar h\upharpoonright g(k)=f_{\ell_k}(n_k)\upharpoonright g(k)$.
    \item For all $k'>k, \,B_{k'}\cap 2^{\leq g(k)}\subseteq \bigcup_{i\leq k+1} B_i$.
    \item For all $t\in 2^{\leq k}$, there is $t'\supseteq t$ such that $t'\in 2^{\leq g(k)}$ and $ t'\notin \bigcup_{i\leq k}B_i$.
\end{enumerate}

Let us make some remarks: $S=\{\sigma: \sigma\subset  stem(T)\}\cup \{\sigma^S_{k}:k<\omega\}$ is, in fact, a condition because $succ_S(\sigma^S_k)=X_k\in\id^+$ for $k<\omega$. The succession mention in (iv) with property (5) can be obtained using the pure decision property. By conditions (3) and (5) we see that $S^{k+1}_{\sigma^{S}_{k+1}}$ decides $\dot h\upharpoonright g(k)$. Conditions (1) and (2) are given by properties (a) and (b). By (2), for all $k$ we have $n_k\in succ_{S^k}(\sigma^{S^k}_{\ell_k})$. We do not specify $n_k$ explicitly; it is determined by (2) and the enumeration of $S^k$.




We start the recursion. Basic step $k=0$. We define $f_0':succ_T(stem(T))\rightarrow2^\omega$ as in (6).
Since $f_0'$ is not a Kat\v{e}tov function, there is $X_0\in \id^+\upharpoonright succ_T(stem(T))$ such that $f_0'[X_0]\subseteq [B_0]$ for some $B_0$ a nowhere dense tree. Take $f_0=f_0'\upharpoonright X_0$. Choose $t'$ such that $t'\not \in B_0$ and define $g(0)=|t'|$.

Recursion step. Assume that our construction has been done for all $i\leq k$. Now do it for $k+1$. Define $f'_{k+1}:succ_{T^{g(k)}(\ell_k,n_k)}(\sigma^S_{k+1}) \rightarrow2^\omega$ as in (6). Since $f_{k+1}'$ is not a Kat\v{e}tov function there is $X_{k+1}\in \id^+$ and $B_{k+1}'$ a nowhere dense tree such that if we define $f_{k+1}=f'_{k+1}\upharpoonright X_{k+1}$, then $f_{k+1}[X_{k+1}]\subseteq B_{k+1}$. We can prune $B_{k+1}'$ to $B_{k+1}$ such that if $t\in B_{k+1}$ has size at most $g(k)$, then there is $n\in X_{k+1}$ such that $f_{k+1}(n)\upharpoonright |t|= t$. Choose $g(k+1)$ in such a way that (11) is satisfied.

Now we need to prove that conditions (9) and (10) hold. To prove (9) take $\bar h\in f_{k+1}[X_{k+1}]$ and $j<g(k)$. Let $m\in X_{k+1}$ be such that $\bar h=f_{k+1}(m)$. By (6), $\bar h(j)=i$ if and only if $T^j(k+1,m)\Vdash``\dot h(j)=i"$. In the same way, $f_{\ell_k}(n_k)(j)=i$ if and only if $T^\ell(\ell_k,n_k)\Vdash``\dot h(j)=i"$.
Notice that by (2),(3) and (5)
\[T^j(k+1,m)\leq S^{k}_{{\sigma^S_{k+1}}^\frown m}=S^{k}_{{\sigma^S_{\ell_k}}^\frown n_k^\frown m}
\]
and  $m\in X_{k+1}\subseteq succ_{T^{g(k)}(\ell_k,n_k)}({\sigma^{S^k}_{\ell_k}}^\frown n_k)$, then we have 
\begin{displaymath}
    T^j(k+1,m)\leq T^{g(k)}(\ell_k,n_k)\leq_0 T^j(\ell_k,n_k).
\end{displaymath}
To prove (10), by the recursion hypothesis we just verify for $k'=k+1$ and $j\leq k$. Let $t\in B_{k+1}\cap2^{\leq g(j)}$. Take $t'\in B_{k+1}\cap2^{g(k)}$ with $t \subseteq t'$. By (9) $t'=f_{\ell_k}(n_k)\upharpoonright g(k)$. Thus $t=f_{\ell_k}(n_k)\upharpoonright |t|$. If $j\geq \ell_k$, $t\in B_{\ell_k}$ and we are done. If $j<\ell_k$ we know that $t\in B_{\ell_k}\cap 2^{g(j)}$. By inductive hypothesis $t\in \bigcup_{i\leq j} B_i$ and we are done.

\begin{claim}
$A=\bigcup\{B_k:k<\omega\}$ is a nowhere dense tree.
\end{claim}
\begin{proofclaim}
    Let $t\in A$, there is $k$ such that $|t|=k$. By (11) there is $t'\in 2^{<g(k)}\setminus (\bigcup_{\ell\leq k}B_{\ell})$. If there is some $k'>k$ with $t'\in B_{k'}$. By (10) there is $\ell\leq k$ with $t'\in B_\ell$. This is an absurd, then $t'\notin B_{k'}$ for all $k'>k$.
\end{proofclaim}

The proof ends when we show that $S\Vdash``\dot h \in [A]$". Searching for a contradiction assume that it does not hold. There are $j<\omega$ and $S'\leq S$ such that $S'\Vdash``\dot h\upharpoonright j\notin A"$. Let $k<\omega$ be such that $stem(S')=\sigma^{S}_k$. Let $k'>k$ be  such that $j<g(k')$ and $\sigma^S_{k'}\in S'$. Then, by (6) and (7),
\begin{displaymath}
    T^j(\ell_{k'}, n_{k'})\Vdash ``\dot h\upharpoonright j\in B_{k'}".
\end{displaymath}
and by(3), 
\[
S'_{\sigma^S_{k'}}\leq S_{\sigma^S_{k'}}\leq S^{k'+1}_{\sigma^S_{k'}}=S^{k'+1}_{{\sigma^S_{\ell_{k'}}}^\frown n_{k'}}\leq_0 T^{g(k')}(\ell_{k'},n_{k'})\leq_0  T^j(\ell_{k'}, n_{k'}).
\]
Thus, $S'_{\sigma^S_{k'}}\Vdash`` \dot h\upharpoonright j\in A"$ which is a contradiction.
\end{proof}

This finishes the study of Cohen reals in this section. Now we concentrate on random reals. In particular, we show that forcing notions of the form $\li$ do not add random reals. It is already known when $\id^+$ is an ultrafilter, since $\mathbb{L}(\id^+)$ is $\sigma$-centered.

Recall that given $f$ in $\omega^\omega$, we say that $f$ is a \emph{bounded eventually different real over} $V$ if there is $g\in V\cap \omega^\omega$ such that for all $n$, $f(n)\leq g(n)$ and for all $h\in V\cap \omega^\omega$, $h\cap f$ is finite (see \cite{barto}).

\begin{theorem}\label{bed}
    Let $\id$ be an ideal on $\omega$. $\li$ does not add bounded eventually different reals.
\end{theorem}
\begin{proof}
Let $T \in \li$, $f \in \omega^\omega$, and $\dot{r}$ an $\li$-name for a function such that 
$T \Vdash ``\dot{r} \leq f"$. Fix an enumeration $\{\tau_n : n < \omega\}$ of the nodes of $T$ below $stem(T)$.

We construct $T' \leq_0 T$ recursively on the levels of $T \setminus stem(T)$, 
with the goal that whenever $\tau_m \in T'$, the subtree $T'_{\tau_m}$ decides the value of $\dot{r}(m)$. 

At stage $0$, by the pure decision property, obtain $T'_0 \leq T_{stem(T)}$ such that $T'_0$ decides $\dot{r}(0)$. 
Assume that $T'$ has been constructed up to level $n$. 
For each $m < \omega$ such that $\tau_m$ lies at level $n$ of $T'$, 
the pure decision property yields a subtree $T'_m \leq_0 T_{\tau_m}$ deciding $\dot{r}(m)$. 
Then, we set $succ_{T'}(\tau_m) = succ_{T'_m}(\tau_m)$.

Once $T'$ is completely constructed, define a function $h \in \omega^\omega$ by letting $h(m) = k$ 
if and only if $T'_{\tau_m} \Vdash ``\dot{r}(m) = k"$. 
By construction, we have $T' \Vdash ``|\dot{r} \cap h| = \omega".$
\end{proof}

It is well known that adding a random real adds a bounded eventually different real.
In this way, we conclude the following:

\begin{coro}
    Let $\id$ be an ideal on $\omega$. $\li$ does not add random reals.
\end{coro}

With this result, we finish this section.

\section{When does $\li$ have the Laver property?}

In this section, we investigate when the forcing $\mathbb{L}(\mathcal{I}^+)$ has the Laver property. While in Section~2 we characterized the addition of Cohen reals via the Kat\v{e}tov order, now our object is to describe the Laver property in a similar framework.

We begin by recalling the definition of the Laver property and then proceed to define the ideals $\mathcal{L}_f$ for a certain type of functions $f$, proving the main theorem of the section, which provides a necessary and sufficient condition for $\mathbb{L}(\mathcal{I}^+)$ to have the Laver property. 
We establish some properties of ideals of the form $\lf$ and show that they are all Kat\v{e}tov equivalent to each other.


In contrast to Theorem \ref{FZ}, we show that not adding Cohen reals is not equivalent to having the Laver property to forcing notions of the form $\li$. In fact, we find an $F_\sigma$ ideal $\id$ such that $\mathbb{L}(\mathcal{I}^+)$ neither has the Laver property nor add Cohen reals. In addition, it is not equivalent even for $P$-points\footnote{ Recall that an ultrafilter is a $P$-point if every function on $\omega$ is finite-to-one or constant when it is restricted to some set in the ultrafilter. For more background see \cite{ultra}.}.

Finally, we extend the main result of this section to forcing notions to the form $\mil(\id^+)$. Also, for forcings of the form $\mil(\id^+)$, not adding Cohen reals is not equivalent to having the Laver property.

\begin{definition}
    Given $g\in \omega^\omega$, $F:\omega\rightarrow [\omega]^{<\omega}$ is a $(g)$-\emph{slalom} if for all $m<\omega$ $|F(m)|\leq g(m)$.
\end{definition}

\begin{definition}\label{LP}
    A forcing notion $\mathbb{P}$ has the \emph{Laver property} if for every $p\in \mathbb{P}$ and every $\mathbb{P}$-name $\dot{f}$ for a real in $\omega^\omega$ such that $p\Vdash ``\dot f\leq g"$ for some $g$ in the ground model, there exist a condition $q\leq p$ and a $(2^n)$-slalom $F$ in the ground model such that:
\[
q \Vdash ``\forall n\,\dot{f}(n) \in F(n)".
\]
\end{definition}

In the previous definition, it is equivalent:
\begin{enumerate}
    \item $F$ is a $(2^n)$-slalom;
    \item there is $h\in \omega^\omega$ a growing function such that $F$ is a $(h)$-slalom.
\end{enumerate} 

It is important to recall that $h$ cannot depend on $g$ or $\dot f$.

Now we introduce a family of ideals that captures the idea of the Laver property. They will be fundamental in the main theorem of this Section.

\begin{definition}
    Let 
    \[
    \mathcal{F}=\{f\in\omega^\omega: f \text{ is an increasing function and } \, f(n)\geq n^2 \text{ for all }n<\omega\}.
    \]
    Fix $f\in\mathcal{F}$. Define:
    \begin{enumerate}
        \item $\mathcal{P}_f=\{s\in \omega^{<\omega}: \forall i\in dom(s)\, (s(i)\leq f(i))\}$,
        \item Given $Z$ an $(n+1)$-slalom $\mathcal{C}_f(Z)=\{s\in \mathcal{P}_f:\forall i\in dom(s)\, (s(i)\in Z(i))\}$ and
        \item $\mathcal{L}_f$ is the ideal generated by $\{\mathcal{C}_f(Z): Z \text{ is an }(n+1)\text{-slalom}\}$.
    \end{enumerate}
\end{definition}

\begin{proposition}
    For all $f\in \mathcal{F}$, $\mathcal{L}_f$ is a proper tall ideal.
\end{proposition}
\begin{proof}
    Let $f\in\mathcal{F}$ and $n<\omega$. Take $Z_0,...,Z_n$ $(m+1)$-slaloms. Since $m^2\leq f$, there is $m<\omega$ such that $(n+1)(m+1)<f(m)$. Choose $s\in \mathcal{P}_f$ of length $m+1$ such that $s(m)\notin \bigcup_{i\leq n} Z_i(m)$. Thus $\mathcal{P}_f\in \lf$.

    Now, to see that $\lf$ is tall, we take $X\subseteq \mathcal{P}_f$ an infinite set. Since $\mathcal{P}_f$ is a finite branching tree, by König's lemma there is an infinite branch contained in $X$. Any branch belongs to $\mathcal{P}_f$.
\end{proof}

\begin{theorem}\label{Laver property}
    The forcing notion $\mathbb{L}(\mathcal{I}^+)$ has the Laver property if and only if for all $f\in\mathcal{F}$, for all $X\in\mathcal{I}^+$, $\mathcal{L}_f\not\leq_{K}\mathcal{I}\upharpoonright X$.
\end{theorem}

\begin{proof}
$(\Rightarrow)$ By contrapositive. Let $X\in \id^+$, $\dot r_{gen}$ be the $\li$-name for the generic real and $\varphi:X\rightarrow \mathcal{P}_f$ be a Kat\v{e}tov function that witnesses $\mathcal{L}_f\leq_{K}\id\upharpoonright X$. Let $\dot h$ be the $\li$-name for the function
\[
\varphi(\dot r_{gen}(0))^{\frown}\varphi(\dot r_{gen}(1))^{\frown}\cdots.
\]
Also let $p=X^{<\omega}$. Fix $Z$ an $(n+1)$-slalom and $T\leq p$. Notice that $p\Vdash``\dot h\leq f"$. We need to prove that there is $T'\leq T$ such that $T'\Vdash``\exists m<\omega (\dot h(m)\notin Z(m))"$. Let $\sigma =stem (T)$ and $k=|\sigma|$. Notice that if 
\begin{displaymath}
\sum_{i<k}|\varphi(\sigma(i))|=j
\end{displaymath}
for some $j<\omega$, there are  $(n+1)$-slaloms $Z_0,...,Z_j$ such that for all $ m<\omega$ $Z(m+j)\subseteq \bigcup \{Z_i(m):i\leq j\}$. Since $\varphi(succ_T(\sigma))$ is positive, there is $n\in succ_T(\sigma)$ such that $\varphi(n)\notin \bigcup \{\mathcal{C}(Z_i):i\leq j\}$. Thus, there is $\ell$ such that $\varphi(n)(\ell)\notin \bigcup_{i<j}Z_i(\ell)$. Therefore,
\begin{displaymath}
    T_{\sigma^\frown n}\Vdash `` \dot h(j+\ell)=(\dot h\upharpoonright j ^\frown \varphi(n))(j+\ell)=\varphi(n)(\ell)\notin \bigcup_{i\leq j}Z_i(\ell)".
\end{displaymath}   
In particular $T_{\sigma^\frown n}\Vdash `` \dot h(j+\ell)\notin Z(j+\ell)"$.

$(\Leftarrow)$ Fix $T\in \li$, $f\in\omega^{\omega}$, and $\dot h$ a $\li$-name for a function such that $T\Vdash``\dot h\leq f"$. Increasing $f$ if necessary, we can take $f\in\mathcal{F}$. We can assume that $T$ decides $\dot h(0)$. Without loss of generality, we can assume that $T$ consists of increasing functions. Let $r\in\mathcal{P}_f$ be the longest function such that $T\Vdash`` \dot h\upharpoonright|r|=r"$.

Recursively, for each $\tau\in T$ with $stem(T)\subseteq \tau$, we construct:
\begin{enumerate}[(i)]
    \item $T^\tau\leq_0 T_\tau$,
    \item $X_{\tau}\subseteq succ_T(\tau)$,
    \item for all $m\in succ_T(\tau) $ a function $t_{\tau^\frown m}\in \omega^{<\omega}$,
    \item $Z_\tau$ an $(n+1)$-slalom and
    \item $\varphi_\tau:succ_{T^\tau}(\tau)\rightarrow\mathcal{P}_f$ a function,
\end{enumerate}
such that
\begin{enumerate}
    \item for all $m\in succ_{T^\tau}(\tau)$, $ T^\tau_{\tau^\frown m}\Vdash``\dot h \upharpoonright m+1 = t_{\tau^\frown m}" $,
    \item if $\sigma\leq_T \tau$, then $T^\sigma\leq T^\tau$,
    \item $X_\tau\in \id^+$,
    \item $\varphi_\tau(m)=t_{\tau^\frown m}$ and
    \item $\varphi_\tau[X_\tau]\subseteq \mathcal{C}_f(Z_\tau)$.
\end{enumerate}

Conditions (1),(3),(4) and (5) can be done in the same way for all nodes of $T$, therefore, we only do it for the stem of $T$ which we call $\sigma_0$: Start with $T=T^{\sigma_0}$. For each $m\in succ_T(\sigma_0)$, using the pure decision property, we can find $T^m\leq_0T_{\sigma_0^\frown m}$ and $t_{\sigma_0^\frown m}$ such that (1) holds. Define $\varphi_{\sigma_0}$ as in (4). Since $\varphi_{\sigma_0}$ is not a Kat\v{e}tov function, there are $X_{\sigma_0}$ and $Z_{\sigma_0}$ as in (3) and (5). To get (2) just do this process by levels.

It is important to note that if $\sigma\leq_T \tau$, then $t_\tau (i)=t_\sigma(i)$ for all $i<max(\tau)$.

Recursively, we can get $S\leq_0 T$ so that for all $\tau\in S$ we have $succ_S(\tau)=X_\tau$. 

Let $\{\sigma_k:k<\omega\}$ be  an enumeration of the splitting nodes of $S$ such that:
\begin{enumerate}
    \item $\sigma_k<\sigma_\ell$ implies $k<\ell$ and
    \item if $s$ is such that $\sigma_k=\sigma^\frown n$, $\sigma_\ell=\sigma^\frown m$ and $n<m$, then $k<\ell$.
\end{enumerate}

Moreover, we can use the pure decision property to prune $S$ so that for each $\ell<\omega$ the condition $S_{\sigma_\ell}$ decides $\dot h\upharpoonright \ell+1$ and $\min(succ_S(\sigma_\ell))>\ell$.

We define a $((n+1)^2)$-slalom as follows: for each $m<|r|$, let $F(m)$ be $\{r(m)\}$. For $m\geq |r|$, let $F(m)= \bigcup\{ Z_{\sigma_k}(m): k\leq m\}$.

We claim that $S\Vdash`` \forall m \, \dot h(m)\in F(m)$". To prove it, proceed by contradiction. Let $S'\leq S$ and $m\in \omega$ be such that $S'\Vdash`` \dot h(m)\notin F(m)"$. Notice that $m>|r|$. Let $\sigma '=stem(S')$.

We distinguish two cases:

Case $S_{\sigma'}$ decides $\dot h(m)$. Let $j$ be the largest $j$ such that $S_{\sigma'\upharpoonright j}$ does not decide $\dot h(m)$. Let $\ell$ be such that $\sigma'\upharpoonright j=\sigma_\ell$. Then $\ell<m$, but 
\begin{displaymath}
 S'\leq S_{\sigma'\upharpoonright j+1}\Vdash ``\dot h(m)\in Z_{\sigma'\upharpoonright j}(m)\subseteq F(m)"
\end{displaymath}
that is a contradiction.

Case $S_{\sigma'}$ does not decide $\dot h(m)$. Choose $k\in succ_{S'}(\sigma')$  with $m\leq k$ and let $\ell$ be such that $\sigma'=\sigma_\ell$. Then $\ell<m$ and
\begin{displaymath}
   S'_{\sigma'^\frown k}\leq S_{\sigma'^\frown k}\Vdash ``\dot h(m)\in Z_{\sigma'}(m)\subseteq F(m)". 
\end{displaymath}
Again, we get a contradiction.
\end{proof}

Recently and independently, Horvath and Özalp in \cite{laverultrafilters} proved a characterization of when $\mathbb{L}(\mathcal{U})$ has the Laver property in the case where $\mathcal{U}$ is an ultrafilter.

Now, we study briefly the ideals $\lf$ and get some corollaries from Theorem \ref{LP}.

\begin{proposition}\label{n2max}
    For all $f,g\in \mathcal{F}$, if $f\leq g$, $\mathcal{L}_g \leq_K \mathcal{L}_f$.
\end{proposition}

\begin{proof}
    Since $f(n)\geq n^2$ for all $n<\omega$, we have $\mathcal{P}_f\in\mathcal{L}_g^+$. Even more, we can verify that $\mathcal{L}_f = \mathcal{L}_g\upharpoonright \mathcal{P}_f$.
    
\end{proof}


\begin{proposition}
    For all $f\in\mathcal{F}$, $\mathcal{L}_f$ is Kat\v{e}tov below of $\nwd$.
\end{proposition}

\begin{proof}
   Let $f \in \mathcal{F}$. Define $\varphi:\omega^{<\omega} \to \mathcal{P}_f$ recursively as follows. Set $\varphi(\emptyset)=\emptyset$. Suppose that $s$ has length $n$ and that $\varphi(s)$ has already been defined. For each $k<\omega$, define
\[
\varphi(s^\frown k)=\varphi(s)^\frown k',
\]
where $k' < f(|s|)$ is the unique integer such that $k \equiv k' \pmod{f(|s|)}$.

To see that this function is a Kat\v{e}tov witness, fix an $(n+1)$-slalom $Z$ and take $s \in \varphi^{-1}(\mathcal{C}_f(Z))$. We may assume that $|s|>0$. Let $k<f(|s|)$ be such that $k\notin Z(|s|)$. Then every $t$ extending $\varphi(s)^\frown k$ does not belong to $\mathcal{C}_f(Z)$. Therefore,
\[
\langle s^\frown k\rangle \cap \varphi^{-1}(\mathcal{C}_f(Z))=\emptyset.
\]    
\end{proof}

To find an ideal $\id$ such that $\li$ does not have the Laver property nor add Cohen reals, the following must hold:
\begin{enumerate}
    \item for some and $X\in \id^+$, $\lf\leq_K \id\upharpoonright X$ and
    \item for all $X\in \id^+$, $\nwd\not\leq_K \id\upharpoonright X$.
\end{enumerate}

The next proposition points out that  $\lf$ is a good candidate for $\id$. This proof is inspired by the proof of Theorem 5.3 in \cite{fullmil}.

\begin{proposition}\label{nwd vs lf}
     For all $f\in \mathcal{F}$, $\nwd\not \leq_K \lf$.
\end{proposition}

To prove this, we need some preliminary lemmas. In particular, the following one corresponds to Theorem 1.4 in \cite{rationals}, which, as the authors note, is a reformulation of a result from \cite{Keremedis}.

\begin{lemma}[Keremedis]\label{kere}
    \[\add(\mathcal{M})=\min \{|\mathcal{A}|: \mathcal{A}\subseteq \nwd \wedge \forall D \text{ dense subset of }\mathbb{Q}\, \exists A\in\mathcal{A} (A\cap D \not=^*\emptyset)\}. \]
\end{lemma}


The following lemma is a direct consequence of Lemma \ref{kere} and the fact that $\mathbb{Q}$ as topological space has a countable base.

\begin{lemma}\label{swd} $\add(\mathcal{M})$ is equal to
    \[\min \{|\mathcal{A}|: \mathcal{A}\subseteq \nwd \wedge \forall D \text{ somewhere dense subset of }\mathbb{Q}\, \exists A\in\mathcal{A} (A\cap D \not=^*\emptyset)\}.\]
\end{lemma}

\begin{lemma}(Batroszynski,\cite{barto})\label{barto cov}
    Let $\mathcal{A}\subseteq \omega^\omega$. The following are equivalent:
    \begin{enumerate}
        \item $|\mathcal{A}|<\add(\mathcal{M})$,
        \item there is $c\in\omega^\omega$ such that for all $f\in \mathcal{A}$ for infinitely many $n$ holds $f(n)=c(n)$.
    \end{enumerate}
\end{lemma}

Inspired by Lemma 5.5 in \cite{fullmil}, we have the following lemma.

\begin{lemma}\label{a<c}
    If $\add(\mathcal{M})<\cov(\mathcal{M})$, then $\nwd\not \leq_K \lf$. 
\end{lemma}
\begin{proof}
        Proceeding by contradiction, assume that there is $\varphi:\mathcal{P}_f\rightarrow \mathbb{Q}$ a Kat\v{e}tov function. Let $\{N_\alpha:\alpha\in \add{(\mathcal{M})}\}\subseteq\nwd$ be a  family such that for all $D\subseteq \mathbb{Q}$ somewhere dense there is $\alpha<\add{(\mathcal{M})}$ such that $D\cap N_\alpha$ is infinite. Thus,
    \begin{displaymath}
        \{\varphi^{-1}[N_\alpha]:\alpha<\add{(\mathcal{M})}\} \subseteq \lf.
    \end{displaymath}
    For each $\alpha$, let $n_\alpha<\omega$, and let $Z_0^\alpha, \dots, Z_{n_\alpha}^\alpha$ be $(n+1)$-slaloms such that
    \begin{displaymath}
        \varphi^{-1}[N_\alpha] \subseteq \bigcup_{i\leq n_\alpha} \mathcal{C}_f(Z_i^\alpha).
    \end{displaymath}
    Let $\rho:\omega^{\leq \omega}\rightarrow\omega^{\leq \omega}$ be a function such that $\rho(g)(n)=g(2n)$. In the same way, let $\bar\rho: ([\omega]^{<\omega})^\omega \rightarrow ([\omega]^{<\omega})^\omega$ be a function defined as follows: $\bar\rho(Z)(n)=Z(2n)$ for  $Z\in([\omega]^{<\omega})^\omega$.

    We consider the family $\{\bar \rho(Z^\alpha_i):\alpha<\add(\mathcal{M})\wedge i\leq n_\alpha\}$. By Lemma \ref{barto cov}, there exists $c:\omega\rightarrow [\omega]^{<\omega}$ such that for all $\alpha<\add(\mathcal{M})$ and $i\leq n_\alpha$, there are infinitely many $n$ such that $c(n)=\bar\rho(Z^\alpha_i)(n)$. Without loss of generality, $c$ is an $(n+1)$-slalom. We define 
    \begin{displaymath}
        T^{c}=\{s\in \mathcal{P}_f: \rho(s)\notin \mathcal{C}_f(\bar\rho(c)) \}.
    \end{displaymath}

    \begin{claim}
        $T^c\in \lf^+$.
    \end{claim}
    \begin{proofclaim}
        Let $Z_0,...,Z_n$ be $(m+1)$-slaloms. Choose $s\in\mathcal{P}_f$ such that for all $i\leq n$ there is $m_i<|s|$ such that $s\notin Z_i(m_i)$. Let $s'\supseteq s$ be such that $|s'|=2|s|+1$ and $s'(2|s|)\not\in c(2|s|)$. Thus, $s'\in T^c$ and it is a witness of $T^c\not\subseteq\bigcup_{i\leq n}\mathcal{C}_f(Z_i)$.
    \end{proofclaim}

    Let $D\subseteq T^c$ be a $(\supseteq)$-dense subset of  $T^c$ such that $|\{s\in D: |s|=2n+1\}|=1$ for all $n<\omega$. By density, $D\in\lf^+$. We claim that $\varphi [D]\cap N_\alpha$ for all $\alpha<\add{\mathcal{M}}$ is finite. Fix $\alpha$, notice that
    \[
    \varphi[D]\cap N_\alpha =\varphi[D\cap \varphi^{-1}[N_\alpha]]\subseteq \varphi[D\cap \bigcup_{i\leq n}\mathcal{C}_f(Z_i^\alpha)] = \varphi[\bigcup_{i\leq n}(D\cap \mathcal{C}_f(Z_i^\alpha))].
    \]

    Now, fix $i\leq n_\alpha$. It is enough to see that $D\cap \mathcal{C}_f(Z_i^\alpha)$ is finite. By the property of $c$ there is $n$ such that $c(n)=\bar\rho(Z^\alpha_i)(n)=Z^\alpha_i(2n)$. On the other hand, by the definition of $T^c$,if $s\in D\subseteq T^c$ and $|s|\geq2n+1$, then $s(2n)\not\in c(n)=Z_i^\alpha(2n)$. Therefore, $s\notin \mathcal{C}_f(Z_i^\alpha)$. Thus, $D\cap \mathcal{C}_f(Z_i^\alpha)$ contains only functions of length at most $2n$. The contradiction comes from the fact that $\varphi[D]$ is somewhere dense.

\end{proof}

\begin{proof}[Proof of Proposition \ref{nwd vs lf}.]
    Since $\nwd$ is a Borel ideal and $\lf$ is analytic, the statement 
    \[\forall \varphi:\mathcal{P}_f\rightarrow \mathbb{Q} \exists A\in\nwd (\neg (\varphi^{-1}[A]\in \lf))\]
    is a $\Pi^1_3$ statement. We call it $\phi$.
    Let $G$ be a generic filter for the product of $\omega_2$ many copies of Cohen forcing after forcing $\CH$. In $V[G]$, $\add(\mathcal{M})<\cov(\mathcal{M})$ holds. By Lemma \ref{a<c}, $V[G]\models \phi$. By the downward absolute between forcing extensions (Shoenfield's absolute theorem), we have $V\models \phi$.
\end{proof}

Unfortunately, we do not know if $\nwd\not\leq_K \lf\upharpoonright X$ for all $X\in \lf^+$ and all $f\in \mathcal{F}$. However, for all $f$ there is a restriction of $\lf$ with this property. Fix $f\in \mathcal{F}$ and define
\[
Y=\{s\in\mathcal{P}_f: \forall i< |s|-1 (s(i)=0 \wedge s(|s|-1)\leq |s|(|s|-3)) \}.
\]

Notice that $Y\in\lf^+$, even more $Y$ satisfies the next property: for all $n<\omega$ there is $m<\omega$  such that $|\{s(m):m<|s| \wedge s\in Y\}|=n(m+1)$. To see this, for  each $n$ take $m=n+2$. Thus, for all $s\in Y$ with $n+2<|s|$, $s(n+2)$ is $0$ when $|s|\neq n+3$. If $|s|= n+3$
\begin{displaymath}
    s(n+2)\leq |s|(|s|-3))=(n+3)(n)=((n+2)+1)n=(m+1)n.
\end{displaymath}
 This property implies that $Y$ cannot be covered with $n$ many $(m+1)$-slaloms for all $n$.

\begin{lemma}\label{positive f sigma}
    For all $X\subseteq Y$, $X\in (\lf\upharpoonright Y)^+$ if and only if for all $n<\omega$ there is $m<\omega$  such that $|\{s(m):m<|s| \wedge s\in X\}|>n(m+1)$.
\end{lemma}
\begin{proof}
$(\Leftarrow)$ Let $Z_0,...,Z_{n-1}$ be $(k+1)$-slaloms, since there are $m$ and\\ $s_0,...,s_{n(m+1)}\in X$ such that  $s_i(m)\neq s_j(m)$ for all $i<j\leq n(m+1)$. Therefore, $\{s_i:i\leq n(m+1)\}\not\subseteq \bigcup_{j<n} \mathcal{C}_f(Z_j)$.

$(\Rightarrow)$ Let $X$ be an element of $(\lf\upharpoonright Y)^+$. By contradiction, assume that there is $n$ such that $|\{s(m):m<|s| \wedge s\in X\}|\leq n(m+1)$ for all $m$. For each $m$ there is $\{Z_i(m):i\leq n\}$ a collection of finite sets such that for all $i\leq n$,
        \begin{enumerate}
            \item $\bigcap Z_i(m)=\{0\}$,
            \item $|Z_i(m)|\leq m+1$ and
            \item $\bigcup_{i\leq n}Z_i(m)=\{s(m):m<|s| \wedge s\in X\}$.
        \end{enumerate}
Therefore, we can define $Z_0,...Z_n$ $(m+1)$-slaloms such that they show $X\in \lf$.
\end{proof}

\begin{lemma}\label{restriccion fsigma}
    $\lf\upharpoonright Y$ is an $F_\sigma$ ideal.
\end{lemma}

\begin{proof}
    Set $\mathcal{A}_n=\{B\in \mathcal{P}_f: B\not\subseteq Y \vee  \exists m \, |\{s(m):s\in B \wedge m<|s| \}|> n(m+1) \}$ for $n<\omega$. By Lemma \ref{positive f sigma}, $\lf\upharpoonright Y= \bigcup_{n<\omega} \mathcal{P}_f\setminus \mathcal{A}_n$. Since both conditions $B\not\subseteq Y$ and $\exists m \, |\{s(m):s\in B \wedge m<|s| \}|> n(m+1)$ can be verified with a finite subset of $B$, we see that  $\mathcal{A}_n$ for all $n$ is an open set. 
\end{proof}

Since there is no $F_\sigma$ ideal above $\nwd$ in Kat\v{e}tov order (\cite{meza}, Corollary 3.3.5), we have the next theorem.

\begin{theorem}
For all $X\in (\lf\upharpoonright Y)^+$, $\nwd \not\leq_K \lf\upharpoonright X$.   
\end{theorem}

As a corollary of the last theorem and Theorem \ref{Cohen reals}, we have the following result:

\begin{coro}\label{diferent}
    $\lav((\lf\upharpoonright Y)^+)$ neither has the Laver property nor adds Cohen reals.
\end{coro}

By Corollary \ref{diferent}, for co-ideals it is not equivalent to have the Laver property and to not add Cohen reals for forcing notions of the form $\mathbb{L}(\mathcal{I}^+)$. Moreover, the same phenomenon occurs for ultrafilters: under $\CH$, every $F_\sigma$ ideal can be extended to a $P$-point that is not Kat\v{e}tov below $\nwd$ (such an ultrafilter is called an $\nwd$-ultrafilter). Thus, if $\mathcal{U}$ is a $P$-point that extends $\lf\upharpoonright Y$, $\nwd\nleq_K \mathcal{U}^*$ and $\lf\upharpoonright Y\leq_K \mathcal{U}^*$. Therefore, we have the following result.

\begin{theorem}
    Assuming $\CH$, there is $\mathcal{U}$ a $P$-point such that $\mathbb{L}(\mathcal{U})$ does not add Cohen reals but it does not have the Laver property.
\end{theorem}

The ideal $\lf$ also characterizes the Laver property for Miller forcing associated with a co-ideal in the same way that Theorem \ref{Laver property}.

\begin{proposition} \label{li implies mi}
    Let $\id$ be an ideal. If $\li$ has the Laver property, then $\mil(\id^+)$ has the Laver property.
\end{proposition}
\begin{proof}
Let $\Phi: \mil(\id^+)\rightarrow \li$ be the function that given $T$ ``collapses"  the nodes  (below $stem(T)$) that have just one immediate successor to the splitting node below them.

Let $T\in\mil(\id^+)$, $F:[T]\rightarrow \omega^\omega$ be a continuous function, $f\in\omega^\omega$ and $\dot r_{gen \mil}$ be the $\mil(\id^+)$-name of the generic real such that $ T\Vdash``F(\dot r_{gen\mil})\leq f".$

Consider $\Phi_T:[T]\rightarrow [\Phi[T]]$ the function induced by $\Phi$. This is an homeomorphism.

We claim that each function $x$ in the image of $F$ is below $f$. If it is not true, there are $y\in[T]$ and $n$ such that $f(n)<F(y)(n)$. Since $F$ is continuous, there is $\tau$ such that $y\in\langle \tau \rangle $ and for all $x$ in $\langle \tau \rangle$ we have $F(x)(n)>f(n)$. Thus, if $S\leq T$ is such that $\tau \subseteq stem(S)$, $S\Vdash`` F(\dot r_{gen\mil})\not\leq f"$, a contradiction.

According to Mostowski’s Absoluteness Theorem, 
\[\Phi(T)\Vdash`` F(\Phi^{-1}_T (\dot r_{gen\lav}))\leq f"\]
where $\dot r_{gen\lav}$ is the canonical name for the generic real added by $\li$. By hypothesis over $\li$, there is $S'\leq \Phi(T)$ and $Z$ an $(m+1)$-slalom such that \[S'\Vdash`` \forall n<\omega (F(\Phi^{-1}_T (\dot r_{gen\lav}))(n)\in Z(n))".\]

As we proved the previous claim, we can show that for each $x$ in the image of $F$ and for each $n<\omega$, $x(n)\in Z(n)$. We can find $S\leq T$ such that $S\in\Phi^{-1}\{S'\}$. Thus, $S\Vdash``\forall n F(\dot r_{gen\mil})(n)\in Z(n)"$.
\end{proof}

The proof of the first part of Theorem \ref{Laver property} can be adapted to show that for all $\id$ an ideal if there is $X\in\id^+$ such that $\lf\leq_K \id\upharpoonright X$, then $\li$ does not have the Laver property. Therefore:

\begin{coro}
    For all $X\in \id^+$, $\mathcal{L}_f\not\leq_{K}\id\upharpoonright X$ if and only if $\mil(\id^+)$ has the Laver property.
\end{coro}

By Theorem \ref{SZ} and Proposition \ref{li implies mi}, Corollary \ref{diferent} holds by replacing $\lav$ by $\mil$, that is:

\begin{coro}
    $\mil((\lf\upharpoonright Y)^+)$ neither has the Laver property nor adds Cohen reals.
\end{coro}

Thus, once again, in contrast to Theorem \ref{FZ}, the notions of having the Laver property and not adding Cohen reals do not coincide for Miller forcing associated with a co-ideal.

\section{When does $\li$ preserves outer measure?}

In this section we adapt the proof of a theorem of Sabok and Zapletal (Theorem \ref{miller medida}) about Miller forcing associated with a co-ideal and outer measure, and we prove the analogous result for $\li$.

To accomplish this goal, we set some notation: $\lambda$ denotes the Lebesgue measure on the Cantor space. Recall that the definition of $\lambda$ for basic open sets is as follows: if $t\in2^{<\omega}$, then $\lambda(\langle t \rangle)=2^{-|t|}$. The outer measure (denoted by $\lambda^*$) of a subset $A$ of $2^{\omega}$ is the infimum of the measures of open sets that cover $A$. We say that a forcing notion $\p$ preserves outer measure if for any $\varepsilon>0$, any $A \subseteq 2^\omega$, and any generic filter $G\subseteq \p$, if $V\models \lambda^*(A)=\varepsilon$, then $V[G]\models \lambda^*(A)=\varepsilon$.

For the rest of this section, we use the following notation. Given $A\subseteq \omega\times 2^\omega$, for all $n< \omega$ we define $A_n=\{x\in2^\omega:(n,x)\in A\}$ and for all $x\in 2^\omega$ we define $A^x=\{n<\omega: (n,x)\in A\}$.

\begin{definition}
An ideal $\id$ on $\omega$ has the \textit{Fubini property} if for all $\varepsilon>0$ and for every Borel set $A\subseteq \omega\times 2^\omega$, 
\[
\{n: \lambda(A_n)\geq\varepsilon\}\in \id^+ \text{ implies } \lambda^*(\{x: A^x\in \id^+\})>\varepsilon.
\]
\end{definition}

Regarding the Fubini property, there are two important results. The first is a characterization of the Fubini property whose proof can be found in \cite{meza}.

\begin{theorem}[Solecki]
Let $\id$ be an ideal on $\omega$. The ideal $\id$ has the Fubini property if and only if there is no $X\in \id^+$ such that $\mathcal{S}\leq_K \restr{\id}{X}$.
\end{theorem}

Using this theorem, the authors of \cite{sabok} proved the following result. Remember that an ideal is \textit{universally measurable} if it is measurable with respect to the completion of every finite 
$\sigma$-additive Borel measure on $2^\omega$. 

\begin{theorem}[Sabok, Zapletal]\label{miller medida}
Assume that $\id$ is a universally measurable ideal on $\omega$. Then $\mil(\id^+)$ preserves outer Lebesgue measure if and only if $\id$ has the Fubini property.
\end{theorem}

\begin{lemma}\label{li medida implica mi medida}
Let $\id$ be an ideal on $\omega$. If $\li$ preserves outer measure, then $\mil(\id^+)$ preserves outer measure.
\end{lemma}

\begin{proof}
Let $A\subseteq 2^\omega$ with measure $\delta$ and $T\in\mil(\id^+)$. We will find $R\leq T$ such that $R\Vdash``\lambda^*(A)=\delta"$. Take $0<\varepsilon<\delta$ and let $\dot U$ be a $\mil(\id^+)$-name for an open set such that $T\Vdash`` \lambda(\dot U)\leq \varepsilon"$. By a standard fusion argument, there is $S\leq T$ and a function $h$ from the splitting nodes of $S$ to the topology of the Cantor space such that 
\[
S\Vdash``\dot U=\bigcup \{h(\dot r_{gen\mil(\id^+)}\upharpoonright n+1): \dot r_{gen\mil(\id^+)}\upharpoonright n \text{ is a splitting node}\}".
\]
One can prove that for all $y\in [S]$,
\[
\lambda\Big(\bigcup \{h(y\upharpoonright n+1): y\upharpoonright n \text{ is a splitting node}\}\Big)\leq \varepsilon.
\]

Recall that there is a function $\Phi:\mil(\id^+) \to \li$ that collapses the nodes between splitting nodes. Moreover, given $S$ there is a homeomorphism between $[S]$ and $[\Phi(S)]$ denoted by $\Phi_S$. By Mostowski's Absoluteness Theorem, if $\dot U'$ is a $\li$-name for 
\[
\bigcup \{h(\Phi_S^{-1}(\dot r_{gen\li})\upharpoonright n+1): \Phi_S^{-1}(\dot r_{gen\li})\upharpoonright n \text{ is a splitting node}\},
\]
then $\Phi(T)\Vdash``\lambda(\dot U')\leq \varepsilon"$.

Since $\li$ preserves outer measure, there is $S'\leq \Phi(T)$ and $a\in A$ such that $S'\Vdash ``a\notin \dot U'"$. In fact, for all $y\in [S']$ we have 
\[
y\notin \bigcup \{h(\Phi_S^{-1}(y)\upharpoonright n+1): \Phi_S^{-1}(y)\upharpoonright n \text{ is a splitting node}\}.
\]
Choose $R\in \Phi^{-1}(S')$. Then $R\Vdash``a\in A \setminus \dot U"$. Therefore, $R$ forces that $\lambda^*(A)> \varepsilon$.
\end{proof}

Based on Theorem \ref{miller medida} and its proof, we prove the following. It is important to remark that the universally measurable hypothesis is not required in our work, given the way we establish the Fubini property.

\begin{theorem}
Let $\id$ be an ideal on $\omega$. Then $\li$ preserves outer measure if and only if $\id$ has the Fubini property.
\end{theorem}

\begin{proof}
From Theorem \ref{miller medida} and Lemma \ref{li medida implica mi medida} it follows that if $\id$ does not have the Fubini property, then $\li$ does not preserve outer measure.

Conversely, assume that $\id$ has the Fubini property. To show that $\li$ preserves outer measure, take $A\subseteq 2^\omega$ of measure $\delta$, $T\in\li$, and a $\li$-name $\dot U$ for an open set such that $T\Vdash``\lambda(\dot U)\leq\varepsilon"$ for some $\varepsilon<\delta$. We can find a function $h$ from the splitting nodes of $T$ to basic clopen sets of $2^\omega$ such that
\[
T\Vdash`` \dot U=\bigcup \{h(\restr{\dot r_{gen}}{n}): n>|stem(T)|\}".
\]

For each $n$ there is a barrier $B_n\subseteq T$ such that for all $b\in B_n$,
\[
\lambda\Big(\bigcup\{h(\sigma): \sigma\subseteq b\}\Big) \geq \frac{2^{n+1}-1}{2^{n+1}}\varepsilon.
\]
Extending the stem of $T$ if necessary, we may assume that $B_0=\{stem(T)\}$ and that equality holds when $n=0$.

Let $\{\tau_n:n<\omega\}$ enumerate the splitting nodes of $T$. We build a dense set $D=\{\sigma_n:n<\omega\}$ such that for each $n$ there is $b_n\in B_n$ with $\tau_n,b_n\subseteq \sigma_n$ and $\sigma_0=stem(T)$.

Define $h':D\to open(2^\omega)$ by
\[
h'(\sigma_n)=\bigcup\{h(\sigma):\sigma\subseteq \sigma_n\}.
\]

Let $g:\omega\times \omega\to \omega$ such that if $m\in succ_T(\sigma_n)$, then $\sigma_n^\frown m\subseteq \sigma_{g(n,m)}$ and there is no $k$ with $\sigma_n^\frown m\subseteq \sigma_k \subset \sigma_{g(n,m)}$.

Define $d:D\to Borel(2^\omega)$ by $d(\sigma_0)=h'(\sigma_0)$ and for $n>0$,
\[
d(\sigma_{g(n,m)}) = h'(\sigma_{g(n,m)})\setminus h'(\sigma_n).
\]
Then $\lambda(d(\sigma_{g(n,m)}))<\frac{\varepsilon}{2^{n+1}}$.

For each $n$ define
\[
A(n)=\{(m,x): m\in succ_T(\sigma_n) \wedge x\notin d(\sigma_{g(n,m)})\},
\]
which is Borel.

Observe that for all $n<\omega$ and $m\in succ_T(\sigma_n)$, $A(n)_m=2^\omega\setminus d(\sigma_{g(n,m)})$ which measure is at least $\frac{2^{n+1}-1}{2^{n+1}}\varepsilon$. Therefore, 
\[
\{m:\lambda(A(n)_m)>\frac{2^{n+1}-1}{2^{n+1}}\varepsilon\}=succ_T(\sigma_n)\in\id^+.
\] 
Using the Fubini property, we have for each $n$, 
\[ 
\lambda^*(\{x:A(n)^x\in \id^+ \})>\frac{2^{n+1}-1}{2^{n+1}}\varepsilon.
\] 
Let $E_n=2^\omega\setminus\{x:A(n)^x\in\id^+\}$. Then $\lambda^*(\bigcup_{n<\omega}E_n)\leq \varepsilon$. Now, we can find $z\in Z\setminus \bigcup_{n<\omega}E_n$, so for all $n$, $A(n)^z\in\id^+$. We define $S\in \li$ such that $S\leq T$ and $succ_S(\sigma_n)=A(n)^z$. Thus, $S\Vdash` z\notin \dot U"$.
\end{proof}

\section{Using Laver and Miller forcings associated with a co-ideal to add half Cohen reals.}

We say that $h\in\omega^\omega$ is an \emph{infinitely often equal real} over $V$ if for all $r\in\omega^\omega \cap V$ we have $|h\cap r|=\omega$. In the literature this kind of reals is also called \emph{half Cohen reals}, the reason is the following result that is attributed to Bartoszy\'{n}ski in \cite{fullmil}.

\begin{proposition}
    If $V_0\subseteq V_1\subseteq V_2$ are models of the Set Theory such that $V_1$ has a half Cohen real over $V_0$ and $V_2$ has a half Cohen real over $V_1$, then there is a Cohen real over $V_0$ which belongs to $V_2$.
\end{proposition}

Clearly, every Cohen real is also a half Cohen real, but the converse fails. This naturally raises the question of whether there exists a forcing notion that adds a half Cohen real while adding no Cohen reals. Several works, including \cite{fullmil} and \cite{hc}, mention that this question was originally posed by Fermlin in a 1996 problem list. Zapletal gave a positive answer in Theorem~1.3 in \cite{hc}. However, the forcing he constructed relies on tools from dimension theory. For this reason, in the same paper he asked whether there exists a forcing that adds half Cohen reals but not Cohen reals, and whose definition does not involve topological dimension.

In \cite{fullmil}, the authors approach this question by considering the forcing of infinitely often equal trees introduced in \cite{ioe}.

\begin{definition}[Spinas]
    Given a tree $T\subseteq\omega^{<\omega}$  is an \emph{infinitely often equal tree} if for all $t\in T$ there is $N>|t|$ such that for all $k<\omega$ we can find $s\in T$ with $t\subseteq s$ and $N<|s|$ that $s(N)=k$.
\end{definition}

This notion generalizes the concept of a full-splitting Miller tree. In \cite{fullmil}, $\mathbb{IE}$ denotes the forcing of infinitely often equal trees.

It is easy to see that $\mathbb{FM}$ adds a Cohen real, while $\mathbb{IE}$ adds a half Cohen real. In \cite{fullmil}, the authors identify a set of conditions in $\mathbb{IE}$ that force the addition of a Cohen real, but it could be a non-dense set in $\mathbb{IE}$. They note that it is unknown whether there is another condition that forces that no Cohen reals are added.

The work presented in this section is motivated by several results from \cite{fullmil}. Our goal is to find an ideal $\hc$ such that $\mil(\hc^+)$ adds half Cohen reals but does not add Cohen reals. Unfortunately, this goal has not yet been achieved.

\begin{definition}
    Given $r\in \omega^\omega$, we define $H_r=\{s\in\omega^{<\omega}:s\cap r =\emptyset\}$ and let $\hc$ be the ideal generated by the family $\{H_r:r\in \omega^\omega\}$.
\end{definition}

Notice that each infinitely often equal tree is an $\hc$-positive set.

\begin{theorem} \label{hc reals}
    Let $\id$ be an ideal and $X\in \id^+$ such that $\hc\leq_K \restr{\id}{X}$. Then $\li$ and $\mil(\id^+)$ add a half Cohen real.
\end{theorem}
\begin{proof}

The following argument applies to both forcing notions; the only difference lies in where the trees are taken from.

Fix $\id$ and $X$ as above, and let $\varphi:X\rightarrow\omega^{<\omega}$ witness that $\hc\leq_K \restr{\id}{X}$. Set $p=X^{<\omega}$. Let $\dot h$ be a name for the function 
\[
\varphi(\dot r_{gen}(0))^{\frown}\varphi(\dot r_{gen}(1))^{\frown}\cdots.
\]
We claim that 
\[
p\Vdash ``\dot h \text{ is a half Cohen real over } V." 
\]

Proceed by contradiction. Let $T\leq p$ and $f$ be a function such that 
\[
T\Vdash ``|f\cap \dot h|<\omega".
\]
Then we can find $T'\leq T$ and $m\in\omega$ such that 
\[
T'\Vdash ``\forall n\geq m\, (f(n)\neq \dot h(n))".
\]

By extending $T'$ if necessary, we may assume that $|stem(T')|\geq m$, so that
\[
\varphi[succ_{T'}(stem(T'))]\in \hc^+.
\]
Let $k$ be the length of the largest initial segment of $\dot h$ decided by $T'$. Then $k\geq m$. Define $f'(n)=f(n+k)$. Since 
\[
\varphi[succ_{T'}(stem(T'))]\not\subseteq H_{f'},
\]
there exist $j,i<\omega$ such that $\varphi(j)(i)=f'(i)=f(i+k)$. Therefore,
\[
T'_{stem(T')^\frown j}\Vdash ``\dot h(k+i)
    =(\restr{\dot h}{k}^\frown\varphi(j))(i)
    =\varphi(j)(i)
    =f(i+k),
"
\]
a contradiction.
\end{proof}

In order to relate this to other results in this work, as Theorem \ref{Cohen reals} and Theorem \ref{Laver property}, a natural question is:

\begin{question}
    If $\id$ is an ideal such that for all $X\in\id^+$ we have $\hc\not\leq_K \restr{\id}{X}$, then $\mil(\id^+)$ does not add half Cohen reals?
\end{question}

In contrast with Theorem \ref{Cohen reals} and Theorem \ref{Laver property}, the pure decision property cannot be used in this context. For this reason, the question is formulated in terms of $\mil$ rather than $\lav$.

Notice that even if the answer to this question is affirmative, it does not yield an affirmative answer to Zapletal's question mentioned at the beginning of this section.

Nevertheless, by Theorem \ref{hc reals} and Theorem \ref{Cohen reals}, it suffices to find an ideal $\id$ such that for all $X\in\id^+$,
\[
\hc\leq_K \restr{\id}{X} <_K \nwd.
\]

\begin{proposition}
    $\nwd\not\leq_K \hc$. 
\end{proposition}
\begin{proof}
    As we did in Theorem \ref{nwd vs lf}, we first prove this Theorem under the assumption of $\add(\mathcal{M})<\cov(\mathcal{M})$.

    Proceeding by contradiction, assume that there is $\varphi:\mathcal{P}_f\rightarrow \mathbb{Q}$ a Kat\v{e}tov function. Let $\{N_\alpha:\alpha\in \add(\mathcal{M})\}\subseteq\nwd$ be a  family such that for all $D\subseteq \mathbb{Q}$ somewhere dense there is $\alpha<\add(\mathcal{M})$ such that $D\cap N_\alpha$ is infinite, remember that this family exists by Lemma \ref{kere}. Thus, $\{\varphi^{-1}[N_\alpha]:\alpha<\add(\mathcal{M})\} \subseteq \hc $. For each $\alpha$, let $n_\alpha<\omega$, and let $r_0^\alpha, \dots, r_{n_\alpha}^\alpha\in\omega^\omega$  such that $\varphi^{-1}[N_\alpha] \subseteq \bigcup_{i\leq n_\alpha} H_{r_i^\alpha}$. Let $\rho:\omega^{\leq \omega}\rightarrow\omega^{\leq \omega}$ be a function such that $\rho(g)(n)=g(2n)$.

    We consider the family $\{\bar\rho(r_i^\alpha): \alpha<\add(\mathcal{M}) \wedge i\leq n_\alpha\}$. By Lemma \ref{barto cov}, there is $c:\omega\rightarrow \omega$ such that for all $\alpha<\add(\mathcal{M})$ and $i\leq n_\alpha$, there are infinitely many $n$ such that  $c(n)= \rho(r_i^\alpha)(n)$. We define $T^{c}=\{s\in \omega^\omega: \rho(s)\subseteq c \}$.

    Notice that $T^c\in \hc^+$, even more so that it is a full-Miller tree.
    
    Let $D\subseteq T^c$ be a dense subset of  $T^c$ such that $|\{s\in D: |s|=2n+1\}|=1$ for all $n<\omega$. By density, $D\in\lf^+$. We claim that $\varphi [D]\cap N_\alpha$ for all $\alpha<\add(\mathcal{M})$ is finite. Fix $\alpha$, notice that
    \[
    \varphi[D]\cap N_\alpha =\varphi[D\cap \varphi^{-1}[N_\alpha]]\subseteq \varphi[D\cap \bigcup_{i\leq n_\alpha}H_{r_i^\alpha}] = \varphi[\bigcup_{i\leq n_\alpha}(D\cap H_{r_i^\alpha})].
    \]

    Now, fix $i\leq n_\alpha$. It is enough to see that $D\cap H_{r_i^\alpha}$ is finite. By the property of $c$ there is $n$ such that $c(n)= \rho (r_i^\alpha)(n)=r_i^\alpha(2n)$. On the other hand, by the definition of $T^c$ if $s\in D\subseteq T^c$ and $|s|\geq2n+1$, then $s(2n)= c(n)=r_i^\alpha(2n)$. Therefore, $s\notin H_{r_i^\alpha}$. Thus, $D\cap H_{r_i^\alpha}$ contains only functions of length at most $2n$. The contradiction comes from the fact that $\varphi[D]$ is somewhere dense.

    Recall that the topology of $[\omega^\omega]^{<\omega}$ is such that, for each $n$, the inclusion map $\iota_n : [\omega^\omega]^n \to [\omega^\omega]^{<\omega}$ is continuous. Let
\[
B=\{(R,A): \exists n<\omega \,\forall r\in R \,\forall s\in A \,\forall \ell<|s| \,(s(\ell)\neq r(\ell))\}\subseteq [\omega^\omega]^{<\omega} \times \wp(\omega).
\]
    Then $B$ is $F_\sigma$, and the image of $B$ under the projection onto the second coordinate is $\hc$. Therefore, $\hc$ is analytic.

    As in the proof of Theorem~\ref{nwd vs lf}, the statement $\nwd\not\leq_K \hc$ is $\Pi^1_3$ and downward absolute. Since this statement holds after forcing $\mathrm{CH}$ and adding $\omega_2$ Cohen reals, it also holds in $V$.
\end{proof}


However, $\hc$ does not work for our goal. The next proposition is inspired by Lemma 2.7 in \cite{fullmil} proved, as the authors mention, by Goldstern and Shelah.

\begin{proposition}
    There is $T^{GS}\in\hc^+$ such that $\nwd\leq_K \restr{\hc}{T^{GS}}$.
\end{proposition}

\begin{proof}
    Let $T^{GS}$ be a full-Miller tree such that for each $n$ there is exactly one splitting node of length $n$ and for each $\tau\in T^{GS}$ we have $succ_T(\tau)\in \{\{0\},\omega\}$.
    
    Fix $\{s_n:n<\omega\}$ to be an enumeration of $2^{<\omega}$ such that $s_0=\emptyset$. We define $\varphi:T^{GS}\to2^{<\omega}$ as follows,
    \[
    \varphi^{-1}(\tau)=s_{\tau(0)} \,^\frown s_{\tau(1)} \,^\frown \, ... \,^\frown s_{\tau(|\tau|-1)}.
    \]

    To verify that $\varphi$ is a Kat\v{e}tov function, fix a nowhere dense set $N\subseteq 2^{\omega}$. For each $n<\omega$, there exists $m_n>0$ such that $\varphi(\tau_n)^\frown s_{m_n}\notin N$, where $\tau_n$ is the unique splitting node of length $n$.

    Now, define $r:\omega\rightarrow\omega$ as $r(n)=m_n$.
    
    We claim  $\varphi[N]\subseteq H_r\cap T^{GS}$. If $\sigma\notin H_r\cap T^{GS}$, there exists $n$ such that $\sigma(n)=r(n)>0$, then $\tau_n\subseteq \sigma$. By the definition of  $\varphi$ we have 
        \[
        \varphi(\sigma)\supseteq \varphi(\tau_n)^\frown s_{\sigma(n)}= \varphi(\tau_n)^\frown s_{r(n)}= \varphi(\tau_n)^\frown s_{m_n} \notin N.
        \]

    Thus, $\varphi(\sigma)\notin N$.
\end{proof}

    It is important to note that the tree $T^{GS}$ used in the previous proof is attributed to Goldstern and Shelah; see Lemma 2.17 in \cite{fullmil}. In that reference, the authors show that every condition $T\in\mathbb{IE}$ below $T^{GS}$ forces a Cohen real. They mention that it is not known that the conditions that add Cohen reals are dense in $\mathbb{IE}$.

    In our case, $\hc$-positive sets are not necessarily trees. In fact, there is $X\in \hc^+$ that is an antichain, even more.

    \begin{proposition} \label{nwd vs antich}
        There is $X\in \hc^+$ that is an antichain and $\nwd\not\leq_K\restr{\hc}{X}$.
    \end{proposition}
    \begin{proof}
        We define $X=\{|s|^\frown s:s\in\omega^{<\omega}\}$. 
        The argument of this proof is quite similar to the proof of Proposition \ref{nwd vs lf}. The same absoluteness argument allows us to show just that $\nwd\not\leq_K\restr{\hc}{X}$ holds under $\add(\mathcal{M})<\cov(\mathcal{M})$. In the same way that Lemma \ref{a<c}, we can find $\{N_\alpha:\alpha\in \add(\mathcal{M})\}\subseteq\nwd$ a  family such that for all $D\subseteq \mathbb{Q}$ somewhere dense there is $\alpha<\add(\mathcal{M})$ such that $D\cap N_\alpha$ is infinite. Proceeding by contradiction, assume that there is $\varphi:X \rightarrow \mathbb{Q}$ a Kat\v{e}tov function. Therefore, $\{\varphi^{-1}[N_\alpha]:\alpha<\add(\mathcal{M})\} \subseteq \hc\upharpoonright X$.
        
        For each $\alpha$, let $n_\alpha<\omega$, and let $r_0^\alpha, \dots, r_{n_\alpha}^\alpha\in \omega^\omega$ be such that $\varphi^{-1}[N_\alpha] \subseteq \bigcup_{i\leq n_\alpha} H_{r_i^\alpha}$. Let $\rho:\omega^{\leq \omega}\rightarrow\omega^{\leq \omega}$ be a function such that $\rho(g)(n)=g(2n)$.

        By Lemma \ref{swd} there is $c:\omega\rightarrow \omega$ a Cohen real for the family $\{\rho(r^\alpha_i):\alpha<\add(\mathcal{M}) \wedge i\leq n_\alpha\}$ and define
        \[
        T^{c}=\{s\in X: \forall 0<i<|\rho(s)| \,(\rho(s)(i)=c(i))\}.
        \]

        It is not difficult to notice that given $\{r_i:i<n\}\subseteq \omega^\omega$ there is $s\in T^c$ with $s(0)=2n$ such that $s\notin \bigcup_{i<n}H_{r_i}$.

        Let $\{p_n:n<\omega\}$ be an enumeration of $\{p;\omega\to\omega : dom(p)\subseteq 2\mathbb{N}+1\}$. Naturally, $p;\omega\to\omega$ means that $p$ is a partial and finite function and $2\mathbb{N}+1$ is the set of odd natural numbers. Recursively, we can build $D=\{t_n:n<\omega\}\subset T^c$ such that:
        \begin{enumerate}
            \item $p_n\subseteq t_n$ and
            \item for all $i<n$, $|t_n|>|t_i|$.
        \end{enumerate}

        By $(1)$, $D\in (\restr{\hc}{X})^+$. We claim that $\varphi[D]\cap N_\alpha$ is finite for all $\alpha<\add(\mathcal{M})$. To get this, it is enough to fix $\alpha<\add(\mathcal{M})$ and $i<n_\alpha$ and verify that $D\cap H_{r^\alpha_i}$ is finite. Let $n>0$ be such that $c(n)=\rho(r^\alpha_i)(n)=r^\alpha_i(2n)$. If $s\in D\subseteq T^c$ and $|s|\geq 2n+1$, then $s(2n)=c(n)=r^\alpha_i(2n)$. Therefore $s\notin H_{r^\alpha_i}$.
    \end{proof}

    \section{A particular case of the 1-1 or constant property in Laver forcing associated with a co-ideal.}

In \cite{poke} the authors study the Marczewski ideal associated with certain specific tree forcing notions. One of the tools used in their analysis is the 1-1 or constant property,  since if a forcing has this property, then the corresponding Marczewski ideal has cofinality greater than $\mathfrak{c}$. In particular, they ask whether $\lav((\fin \times \fin)^+)$ has the 1-1 or constant property (see Question 2.15 in \cite{poke}). In this section, we answer this question affirmatively.

\begin{definition}[Brendle, Khomskii and Wohofsky \cite{brendle}]
    Let $\id$ be an ideal on $\omega$. We say that $\li$ has the \emph{1-1 or constant property} if for each $T\in \li$ and each $f:[T]\rightarrow 2^\omega$ there exists $S\leq T$ such that $\restr{f}{[S]}$ is injective or constant.
\end{definition}

Since we are interested in the case of $\fin \times \fin$, in this section we set $\mathcal{I}=\fin \times \fin$ (for short). 

Recall that given $\tau \in 2^{<\omega}$, we denote the set $\{f\in 2^{\omega}: \tau\subseteq f \}$ by $\langle\tau \rangle$.

\begin{lemma}\label{convergencia}
    Let $T\in \li$ and $f:[T]\rightarrow 2^\omega$. For all $\tau\in T$ there are $T'\leq_0T_\tau$, $x_\tau\in 2^{\omega}$ and an increasing sequence $(N_k)_{k\in\omega}$  such that for all $(i,j)\in succ_{T'}(\tau)$, if $i\geq N_k$, then $[T'_{\tau^\frown(i,j)}]\subseteq f^{-1}(\langle \restr{x_\tau }{|\tau|+k}\rangle)$.
\end{lemma}
\begin{proof}
    Using the pure decision property, we build a $\leq_0$-sequence $(S^n)_{n<\omega}$ of elements of $\li$ such that $S^n\Vdash``\restr{f(\dot r_{gen})}{|\tau|+n}=x_n"$ for some $x_n$. Let $x_\tau=\bigcup_{n<\omega}x_n$ and for each $n<\omega$ set $A_n=succ_{S^n}(\tau)$. For each $i<\omega$, we denote the $i$-th infinite column of $A_n$ by $A_n(i)$. 
    Let $B=\bigcup\{A_n(n):n<\omega\}\in\id^+$.
    Define $T'$ as follows:
    \begin{enumerate}
        \item $stem(T')=\tau$,
        \item $succ_{T'}(\tau)=B$ and
        \item for each $(i,j)\in B$, if $n=max\{m:(i,j)\in A_m\}$, then $T'_{\tau^\frown(i,j)}=S^n_{\tau^\frown(i,j)}$.
    \end{enumerate}

    Let $\pi_0$ be the projection on the first input; then we define $N_k=\pi_0(A_k(k))$, that is, the $k$-th column of $B$.
    
    Now, we prove that $S, x_\tau$ and $(N_k)_{k<\omega}$ work. Fix $k<\omega$ and take $(i,j)\in B$ with $i\geq N_k$. There is $n<\omega$ greater than $k$ such that $i=N_n$. Note that, since $(i,j)\in B$ and $(A_m)_{m<\omega}$ is $\subseteq$-decreasing, for all $\ell>n$ we have $(i,j)\notin A_\ell(\ell)$. Thus, $T'_{\tau^\frown(i,j)}=S^n_{\tau^\frown(i,j)}$ and $S^n\Vdash``\restr{f(\dot r_{gen})}{|\tau|+n}=x_n"$. If $y\in [T'_{\tau^\frown(i,j)}]$, since $S^n\Vdash ``f(y)\in\langle x_n\rangle"$ we get $y\in f^{-1}(\langle\restr{x_\tau}{|x_n|}\rangle)=f^{-1}(\langle\restr{x_\tau}{|\tau|+n}\rangle)\subseteq f^{-1}(\langle\restr{x_\tau}{|\tau|+k}\rangle)$.
\end{proof}

\begin{lemma}\label{convergencia fusión}
    Let $T\in \li$ and $f:[T]\rightarrow 2^\omega$. There is $T'\in \li$ with $T'\leq_0 T$ such that for all $\tau\in T'$ there is $x_\tau\in 2^\omega$, and an increasing sequence $(N_k^\tau)_{k<\omega}$ such that if $(i,j)\in succ_{T'}(\tau)$ and $i\geq N_k^\tau$, then $[T'_{\tau^\frown(i,j)}]\subseteq f^{-1}(\langle\restr{x_\tau}{|\tau|+k}\rangle)$.
\end{lemma}
The way to prove this Lemma is using a fusion argument and in each step apply Lemma $\ref{convergencia}$.

As consequences of Lemma \ref{convergencia fusión} we have for each $\tau\in T'$:
\begin{enumerate}
    \item $[T'_\tau]\subseteq f^{-1}(\langle \restr{x_\tau}{|\tau|} \rangle)$
    \item The sequences $(x_{\tau^\frown(i,j)})_{(i,j)\in succ_{T'}(\tau)}$ $(\fin\times \emptyset)$-converge to $x_\tau$, i.e., for all $N<\omega$ there is $m<\omega$ such that for all $i\geq m$, $x_{\tau^\frown(i,j)}\in \langle \restr{x_\tau}{N} \rangle$.
\end{enumerate}

\begin{theorem}
    $\li$ has the one-to-one or constant property. 
\end{theorem}

\begin{proof}
    Let $T\in \mathbb{L}(\mathcal{I}^+)$. By Lemma \ref{convergencia fusión}, there exist $T'\leq_0 T$ and a family $\{x_\tau, (N_k^\tau)_{k<\omega} : \tau \in T'\}$ as specified in the statement of the lemma.

    We define a rank function $\rho:T'\rightarrow\omega_1$ as follows:
    \begin{enumerate}
        \item $\rho(\tau)=0$ if and only if $\{(i,j)\in succ_{T'}(\tau):x_{\tau^\frown(i,j)}\not =x_\tau\}\in\id^+$.
        \item $\rho(\tau)=\alpha$ if $\alpha$ is the least such that $\{(i,j):\rho(\tau^\frown(i,j))<\alpha\}\in\id^+$.
    \end{enumerate}

    Consider two cases.
    
    Case 1: There exists $\tau$ such that $\rho(\tau)$ is not defined (or defined as $\omega_1$). With this assumption, we can get $S\leq_0 T'_\tau$ such that for all $\sigma \in S$, $\rho(\sigma)$ is not defined. In particular, $x_\sigma=x_\tau$  for all $\sigma \in S$ and $\tau\subseteq \sigma$. We claim that $\restr{f}{[S]}$ is constant. Fix $y\in [S]$. By construction,  we know that  $[S_{\restr{y}{k}}]\subseteq f^{-1}(\langle \restr{x_\tau}{k} \rangle)$ for all $k<\omega$. Therefore, $f(y)=x_\tau.$ When we unfix $y$, we prove the claim.

    Case 2: For all $\tau$ such that $\rho(\tau)$ is defined. To build $S$ we will define the next items recursively:
    \begin{enumerate}[(i)]
        \item $\{F_n:n<\omega\}$ antichains of $T'$, 
        \item $\{A_\sigma:n<\omega\wedge \sigma\in F_n \}$ subsets of $T'$,
        \item $S=\{\tau\in T':\exists n<\omega(\tau\in F_n \vee \exists \sigma\in F_n (\tau\in A_\sigma) ) \}$ a subtree of $T'$ and
        \item $\{a_\tau\subseteq x_\tau: \exists n<\omega\,\exists \tau\in F_n  ( a_\tau \text{ is a finite initial segment}) \}$
    \end{enumerate}
    such that if we define $A_\sigma ^0=\{\tau\in A_\sigma:\rho(\tau)=0\}$ for each $n<\omega$ and $\sigma\in F_n$, then:
        \begin{enumerate}
        \item If $\tau\in  A_\sigma^0$, then there is no $\tau'\in A_\sigma$ with $\tau'<\tau$. If  $\tau \in A_\sigma\setminus A_\sigma^0$ and  $\rho(\tau)=\alpha>0$, then there is $\tau'\in A_\sigma$ such that $\tau\subseteq \tau'$ and $\rho(\tau)<\alpha$;
        \item  if $\sigma\in F_n$ and $\tau \in F_{n+1}$ with $\sigma\subseteq \tau$, then $a_\sigma\subseteq a_\tau$; 
        \item for each $\sigma\in F_n$ and $\tau\in A_\sigma^0$, $[S_\tau]\subseteq f^{-1}(\langle a_\tau \rangle)$ and
        \item for all $n<\omega$,  $\sigma,\sigma'\in F_n$, $\tau\in A_\sigma^0$, $\tau'\in A_{\sigma'}^0$, $(i,j)\in succ_S(\tau)$ and $(i',j')\in succ_S(\tau')$, we have that $(\tau\neq \tau')$ implies $\langle a_{\tau^\frown(i,j)} \rangle\cap \langle a_{\tau'^\frown(i',j')} \rangle=\emptyset$.
    \end{enumerate}

Let us first verify that $S$ is the tree we are looking for, and then proceed to construct the objects mentioned above.

Take $y,y'\in [S]$ with $y\neq y'$. Let $n<\omega$ be the maximum that there is $\sigma \in F_n$ such that $\sigma\subseteq y,y'$.  Let $\tau,\tau'\in F_{n+1}$ be such that $\sigma\subseteq \tau\subseteq y$ and $\sigma\subseteq \tau'\subseteq y'$. By (3), $f(y)\in \langle a_\tau\rangle$ and $f(y')\in \langle a_{\tau'}\rangle$. Using the maximality of  $n$, we can find $\mu\in A_\sigma^0 $, $\mu'\in A_{\sigma'}^0$ with $\mu\neq \mu'$ and $(i,j),(i',j')\in \omega\times \omega$ such that $\tau=\mu^\frown(i,j)$ and $\tau'=\mu'\frown(i',j')$. By (4), we are done.

Now we start with the recursive construction. 

    For $n=0$. Let $F_0=\{stem(T)\}$, and $a_{stem(T')}=\restr{x_{stem(T')}}{|stem(T')|}$. The construction of $A_{stem(T')}$ is also recursive. We put $\{stem(T')\}\in A_{stem(T')}$. If $\sigma\in A_{stem(T')}$ and $\rho(\sigma)=0$, then $\sigma$ does not have successors in $A_{stem(T')}$. If $\rho(\sigma)=\alpha>0$, we can assume, pruning $succ_{T'}(\sigma)$ if necessary, that  for all $(i,j)\in succ_{T'}(\sigma)$ $x_{\sigma^\frown (i,j)}=x_\sigma$ and $\rho(\sigma^\frown(i,j))<\rho(\sigma)$. Put $succ_{T'}(\sigma)$ into $A_{stem(T')}$.

    Assume that for some $n$, $F_n$ is defined and for all $\tau\in F_n$, $a_\tau$ and $A_\tau$ are defined. For each $\sigma\in F_n$,  let $(\tau_{n,\sigma})_{n<\omega}$ be an enumeration of $A_\sigma^0$. Note that by construction $x_\sigma=x_{\tau_{n,\sigma}}$. Define for each $m$ and $(i,j)\in succ_{T'}(\tau_{m,\sigma})$ 
    \[
    \Delta(i,j,m,\sigma)= \min\{k:x_{\tau_{m,\sigma}^\frown(i,j)}(k)\not=x_{\sigma}(k)\}.
    \]
    Clearly, this induces a coloring between $\bigcup\{succ_{T'}(\tau_{m,\sigma}):m<\omega \}$ and $\omega$. For simplicity call it $\Delta_{m,\sigma}$.

    Given $X\in\id^+$, we use $supp(X)$ to denote the set $\{i<\omega: \exists^{\infty}j<\omega ((i,j)\in X)\}$ which is infinite. Without loss of generality, we can assume that for each $m<\omega$ one of this condition holds:

    \begin{enumerate}
        \item $\forall (i,j),(\ell,k)\in succ_{T'}(\tau_{m,\sigma})\,(\Delta_{m,\sigma}(i,j)\not =\Delta_{m,\sigma}(\ell,k))$. This means that, $\restr{\Delta_{m,\sigma}}{succ_{T'}(\tau_{m,\sigma})}$ is one to one. 
        \item $\forall i,k\in supp(succ_{T'}(\tau_{m,\sigma}))\forall j,\ell<\omega \, ((k,\ell),(i,j)\in succ_{T'}(\tau_{m,\sigma}))$
        \begin{enumerate}
            \item $i=k \Rightarrow
            \Delta_{m,\sigma}(i,j) =\Delta_{m,\sigma}(k,\ell)$ or
            \item $i\not=k \Rightarrow \Delta_{m,\sigma}(i,j)\not=\Delta_{m,\sigma}(k,\ell)$.
        \end{enumerate}
        This is $\restr{\Delta_{m,\sigma}}{succ_{T'}(\tau_{m,\sigma})}$ colors each column the same color and two elements from different columns have different colors.
    \end{enumerate}

    For each $m<\omega$, we say that $\tau_{m,\sigma}$ has property 1 (P1) if the first condition above holds for $\tau_{m,\sigma}$ and it has property 2 (P2) if the second condition above holds for $\tau_{m,\sigma}$. Also fix $\{r^m_k:k<\omega\}$ an increasing enumeration of $supp(succ_{T'}(\tau_{m,\sigma}))$.

    By recursion over $m$ we built the following:
    \begin{enumerate}
        \item $\{B(i,j,\sigma)\in \id:i\leq m\wedge j=m-i\}\subseteq succ_{T'}(\tau_{i,\sigma})$ and
        \item $\mathcal{Z}_{m,\sigma}=\{Z_\sigma(i,j)\in [2^{<\omega}]^{<\omega} : i\leq m \wedge j=m-i\}$
    \end{enumerate}
    such that for $i<\omega$ and $j\not=j'$, $B(i,j,\sigma)\cap B(i,j',\sigma)=\emptyset$. Also, if $\tau_{i,\sigma}$ has P1, then $B(i,j,\sigma)$ is finite, and if $\tau_{i,\sigma}$ has P2, then $B(i,j,\sigma)$ is contained in a set of the form $\{k\}\times \omega$ . The sets $Z_\sigma(i,j)\in \mathcal{Z}_{m,\sigma}$ will be used to have control over the number of colors used at step $m$.

    For $m=0$. We take two cases:
    Case $\tau_{0,\sigma}$ has P1. Choose $j$ such that $(r_0^0,j)\in succ_{T'}(\tau_{0,\sigma})$ and let $B(0,0,\sigma)=\{(r_0^0,j)\}\in \id$ and $Z_\sigma(0,0)=\{\Delta_{0,\sigma}(r_0^0,j)\}$.
    Case $\tau_{0,\sigma}$ has P2. Let $B(0,0,\sigma)=\{(i,j)\in succ_{T'}(\tau_{0,\sigma}): i=r_0^0\}\in \id$. By P2 if we take $Z_\sigma(0,0)=\{\Delta_{0,\sigma}(i,j):(i,j)\in B(0,0,\sigma)\}$, it has size 1.

    Assume that for all $r\leq m$ the following sets are defined:
    \begin{enumerate}
        \item $\{B(i,j,\sigma)\subseteq succ_{T'}(\tau_{i,\sigma}): i\leq r\wedge j=r-i\}$ and
        \item $\mathcal{Z}_{r,\sigma}=\{Z_\sigma(i,j)\in [2^{<\omega}]^{<\omega} : i\leq r\wedge j=r-i\}$.
    \end{enumerate}
    Now we build for $m+1$: By recursion on $i\leq m+1$.
    Assume that we already built $\{B(s,t,\sigma),Z_\sigma(s,t): s<i \wedge t=m+1-s\}$.

    By hypothesis $\bigcup\{B(i,j,\sigma):j\leq m-i\}\in\id$. Let $W_\sigma(i,m+1)$ be the union of  $\bigcup\{B(i,j,\sigma):j\leq m-i\}$ and
    \begin{displaymath}
             \{(\ell,k): \Delta_{i,\sigma}(\ell,k)\in \bigcup\{\bigcup \mathcal{Z}_{r,\sigma}:r\leq n\}\cup \bigcup\{Z_\sigma(s,t):s<i\wedge t=m+1-s\} \}.
    \end{displaymath}
    By construction, for all $r\in\bigcup \mathcal{Z}_{r,\sigma}$ is finite and also $Z_\sigma(s,t)$ for $s<i$ and $t=m+1-s$ is finite. Then $W_\sigma(i,m+1)\in \id$.
    
    To define $B(i,m+1-i,\sigma)$ and $ Z_\sigma(i,m+1-i)$ we consider two cases:

    Case 1: $\tau_{i,\sigma}$ has P1. Since $\bigcup\{B(i,j,\sigma):j\leq m-i\}$ is finite, $W_\sigma(i,m-i)$ is finite. Then $succ_{T'}(\tau_{i,\sigma})\setminus W_\sigma(i,m+1)\in\id^+ $. So, we can find $j_0,...,j_{m+1}$ such that $\{(r_s^i,j_s): s\leq m+1\}\subseteq succ_{T'}(\tau_{i,\sigma})\setminus W_\sigma(i,m+1)$. 
    
    Let $B(i,m+1-i,\sigma)=\{(r^i_k,j_k):k\leq m+1\}$ (a finite set) and 
    \[ Z_\sigma(i,m+1-i)=\{ \Delta_{i,\sigma}(\ell,k): (\ell,k)\in B(i,m+1-i,\sigma)\}.
    \]
    First, notice that since $\tau_{i,\sigma}$ has P1, $Z_\sigma(i,m+1-i)$ is finite. Next, notice that for all $j$ such that $i+j\leq m$, we have $B(i,j,\sigma)\cap B(i,m+1-i,\sigma)=\emptyset$ by construction.

    Case 2: $\tau_{i,\sigma}$ has P2. Take $s$ the minimum such that 
    \begin{displaymath}
        (\{r_s^i\}\times \omega)\cap(succ_{T'}(\tau_{i,\sigma})\setminus W_\sigma(i,m+1))\not =\emptyset.
    \end{displaymath}
    
    Let $B(i,m+1-i,\sigma)=(\{r^i_s\}\times \omega)\cap succ_{T'}(\tau_{i,\sigma})$ and
    
    \[Z_\sigma(i,m+1-i)=\{\Delta_{i,\sigma}(\ell,k): (\ell,k)\in B(i,m+1-i,\sigma)\}.\]
    Since $\tau_{i,\sigma}$ has P2, $|Z_\sigma(i,m+1-i)|=1$.

    With the last recursion finished, for each $B_m(\sigma)=\bigcup\{B(m,j,\sigma):j<\omega\}\subseteq succ_{T'}(\tau_{m,\sigma})$. To prove that $B_m(\sigma)\in \id^+$, we consider two cases: If $\tau_{m,\sigma}$ has P1, $|B(m,j,\sigma)\cap (\{r^m_k\}\times \omega)|$ is finite for each $k<m+j$. Using the fact that $B(m,j,\sigma)\cap B(m,j',\sigma)=\emptyset$ when $j\not=j'$, we are done. If $\tau_{m,\sigma}$ has P2 $B(m,j,\sigma)$ adds an infinite column to $B_m(\sigma)$.

    We define $F_{n+1}$ as
    \begin{displaymath}
        \{\tau\in T':\exists \sigma\in F_n \exists m<\omega \exists(i,j)\in B_m(\sigma)( \tau=\tau_{m,\sigma}^{\frown}(i,j))\}.
    \end{displaymath}

    Let $\tau\in F_{n+1}$. If $\sigma\in F_n$, $m<\omega$ and $(i,j)\in B_m(\sigma)$ are such that $\tau=\tau_{m,\sigma}^{\frown}(i,j)$, we define $a_\tau=\restr{x_\tau}{\Delta_{m,\sigma}(i,j)}$.

    Now we show that our construction satisfies points (1) to (4).
    (1) is an immediate consequence of our construction.
   (2) follows from the fact that if $\sigma\in F_n$ and $\tau \in F_{n+1}$ with $\sigma\subseteq \tau$, there is $i,j,m<\omega$ such that $\tau=\tau_{m,\sigma}^\frown(i,j)$ and the definition of $\Delta(i,j,m,\sigma)$.
    In order to see (3), by Lemma \ref{convergencia fusión}, for each $n$ and each $\tau\in F_n$ we can prune $succ_{S}(\tau)$ dropping finite many columns to get $[S_\tau]\subseteq f^{-1}(\langle\restr{x_\tau}{|a_\tau|}\rangle)$.
    To see (4) we proceed by induction. For $n=0$, to avoid repetition, this case will be omitted because we use an argument as in case 2 in the inductive step. Let $n>0$  and assume that for all $m\leq n$ the statement is true. To verify for $n+1$ let $\sigma,\sigma'\in F_{n+1}$,  $\tau\in A_\sigma^0$ and $\tau'\in A_{\sigma'}^0$ with $\tau\neq \tau'$. Take $k$ as the maximum number that satisfies that there is $ \mu\in F_k$ such that $\mu\subseteq \tau,\tau'$. Consider two cases. Case 1: $k<n$.  Take $\nu,\nu'\in F_{k+1}$ which extends to $\tau$ and $\tau'$, respectively. By the maximality of $k$, we have $\nu\neq \nu'$. By inductive hypothesis, $\langle a_{\restr{\tau}{|\nu|+1}} \rangle\cap \langle a_{\restr{\tau'}{|\nu'|+1}} \rangle=\emptyset$. Case $k=n$. Then $\sigma=\sigma'$ and  there must exist $m,m'<\omega$ (different) such that $\tau=\tau_{m,\sigma}$ and $\tau'= \tau_{m',\sigma}$. Without loss of generality, suppose $m<m'$, then by construction, for all $(i',j')\in B_{m'}(\sigma)$, $\Delta_{m',\sigma}(i',j')\notin \bigcup\{\bigcup \mathcal{Z}_{r,\sigma}:r< m'\}$. In particular for $r=m$. Then for all $(i,j)\in B_m(\sigma)$ we have $\Delta_{m',\sigma}(i',j')\neq \Delta_{m,\sigma}(i,j)$. This implies  $\langle a_{\tau^\frown(i,j)} \rangle\cap \langle a_{\tau'^\frown(i',j')} \rangle=\emptyset$ by definition. 
\end{proof}

In \cite{poke}, Question 2.15 also asks for $\lav(\mathcal{Z}^+)$. However, we have no answer for this yet.

\section{Some cardinal invariants related to ideals.}

In this section we study several cardinal invariants related to ideals, answering questions posed in \cite{poke}, where the following definitions were introduced.

\begin{definition}[Cie{\'s}lak and Mart{\'\i}nez-Celis]\label{*-omega}
    Given an ideal $\id$,
    \[
    \add_\omega^*(\id)=\min\{ |\mathcal{A}|: \mathcal{A}\subseteq \id \,\forall (X_n)_{n<\omega}\subseteq \id \,\exists A\in \mathcal{A} \,\forall n<\omega\, (|A\setminus X_n|=\omega)\}
    \]

    \[
    \non_\omega^*(\id)= \min\{|\mathcal{A}|: \mathcal A \subseteq [\omega]^\omega\,\forall(X_n)_{n<\omega}\subseteq \id \,\exists A\in \mathcal{A} \,\forall n<\omega \,(A\cap X_n=^*\emptyset) \}
    \]
\end{definition}

For this section, we use some properties of nowhere dense and meager sets that have a better expression in terms of chopped reals (see, for example, \cite{combinatorial}).

\begin{definition}[Talagrand, \cite{talagrand}]
    We use $IP$ to denote the set of interval partitions of $\omega$, that is, $P\in IP$ if $P$ is a partition of $\omega$ of the form $P=(P_n)_{n<\omega}$, where each $P_n$ is an interval. A chopped real is a pair $(x,P)$ where $x\in 2^\omega$ and $P\in IP$.
\end{definition}

We can order $IP$ as follows. $P\leq^* Q$ if for all but finitely many $n$ there is $k$ such that $P_k\subseteq Q_n$.

\begin{theorem}[Theorem 2.10, \cite{cccc}]
    \(\mathfrak{d}\) is the smallest cardinality of any family of interval partitions dominating all interval partitions. \(\mathfrak{b}\) is the smallest cardinality of any family of interval partitions not all dominated by a single interval partition.
\end{theorem}

With this, it is possible to verify the next lemma.

\begin{lemma} \label{d-part}
    Let $\mathcal{D}\subseteq IP$ be a family with $|\mathcal{D}|<\mathfrak{d}$. There is $R\in IP$ such that for all $P\in \mathcal{D}$ there are infinitely many $n$ such that $R_n$ contain an interval of $P$. 

\end{lemma}


Given $(x,P)$ a chopped real, let $N(x,P)=\{y\in2^\omega: \forall n \,(y\upharpoonright P_n\neq x\upharpoonright P_n)\}$. It is a nowhere dense set and countable unions of this type of sets form a base for the meager ideal in $2^{\omega}$. In fact, in \cite{combinatorial} is defined 
\[
Match(X,P)=\{y\in2^\omega: \exists^\infty n \,(y\upharpoonright P_n= x\upharpoonright P_n)\}
\]
and it is proved that:

\begin{lemma}\label{chopped}
    $M\subseteq 2^\omega$ is meager if and only if there is a chopped real $(x,P)$ such that $M\cap Match(X,P)=\emptyset$.
\end{lemma}

In \cite{poke} the authors proved in Theorem 4.9 that $\add^*_\omega(\nwd)\leq \add(\mathcal{M})$ and ask in Question 4.10 if these cardinals are equal.

\begin{proposition}
    $\add(\mathcal{M})\leq \add^*_\omega(\nwd)$.
\end{proposition}
\begin{proof}
    Let $\kappa<\add(\mathcal{M})$ be an infinite cardinal and 
    \[\{N(x^\alpha,P^\alpha): \alpha<\kappa \wedge P^\alpha\in IP \wedge x^\alpha\in 2^\omega\}\subseteq \nwd.\]
    Note that, for each $\alpha<\kappa$, the set 
    \[
    \{ y\in 2^\omega : \forall^\infty n \,(y\!\upharpoonright P_n^\alpha \neq x^\alpha \!\upharpoonright P_n^\alpha)\}
    \]
    can be expressed as a countable union of sets of the form $N(x^\alpha, Q^\alpha)$, where $Q^\alpha$ is obtained by modifying finitely many elements of $P^\alpha$. Take $c\in 2^\omega\setminus \bigcup_{\alpha<\kappa}N(x^\alpha, P^\alpha)$. 

    Since $\kappa$ is infinite, we can assume that, for all $\alpha<\kappa$, there are infinitely many $n$ such that $c\!\upharpoonright P^\alpha_n = x^\alpha \!\upharpoonright P^\alpha_n$. For each $\alpha<\kappa$, we can find $R^\alpha\in IP$ such that, for all $n$, there is $m$ with $P^\alpha_m \subseteq R^\alpha_n$ and $c\!\upharpoonright P^\alpha_m = x^\alpha\!\upharpoonright P^\alpha_m$. 

    Since $\{R^\alpha:\alpha<\kappa\}$ has size $<\mathfrak{b}$, there is $P\in IP$ such that, for all $\alpha <\kappa$, $R^\alpha \leq^* P$.

    Let $\mathcal{P}=\{Q\in IP : Q =^* P\}$ (i.e., there is $n$ such that, for all $k>n$, $P_k = Q_k$), which is a countable set. We will now show that $\{N(c,Q) : Q\in\mathcal{P}\}$ witnesses that $\kappa < \add^*_\omega(\nwd)$.

    Fix $\alpha<\kappa$ and $y\in N(x^\alpha, P^\alpha)$. Thus, for all $n$, $y\!\upharpoonright P_n^\alpha \neq x^\alpha\!\upharpoonright P_n^\alpha$. By the definition of $P$, for all $n$, $y\!\upharpoonright R_n^\alpha \neq x^\alpha\!\upharpoonright R_n^\alpha$. Since $c\notin N(x^\alpha, P^\alpha)$, there are infinitely many $n$ such that $c\!\upharpoonright P_n^\alpha = x^\alpha\!\upharpoonright P_n^\alpha$.  

    Choose $Q\in IP$ such that, for all $n$, there exists $k$ with  $R_k^\alpha \subseteq Q_n$. Then, for each $n$, there is $m$ such that $P_m^\alpha \subseteq Q_n$ and $c\!\upharpoonright P_m^\alpha = x^\alpha\!\upharpoonright P_m^\alpha$. Therefore, for each $n$, there exists $m\in Q_n$ such that  $y(m)\neq x^\alpha(m)=c(m)$. Thus, for each $n$,
    \[
    y\!\upharpoonright Q_n \neq c\!\upharpoonright Q_n.
    \]

\end{proof}

\begin{coro}
     $\add(\mathcal{M}) = \add^*_\omega(\nwd)$.
\end{coro}

Another question that appears in \cite{poke} is Question 4.11, which asks for the value of $ \non^*_\omega(\ed)$.

For the next result, we need to remember a characterization of $\cov(\mathcal{M})$ due to Bartoszy\'{n}ski (Lemma 2.4.2, \cite{barto}).

\begin{theorem}\label{cov m}
    Let $\kappa$ be an infinite cardinal. The following statements are equivalent:
    \begin{enumerate}
        \item $\kappa<\cov(\mathcal{M})$
        \item $\forall F \in [\omega^\omega]^{<\kappa} \, \forall G \in [[\omega]^\omega]^{<\kappa}    \,\exists g \in \omega^\omega\,\forall f \in F\,\forall X \in G\,\exists^\infty n \in X\\(f(n) = g(n))$.
    \end{enumerate}
\end{theorem}

It is easy to note the next lemma from the Definition \ref{*-omega}.

\begin{lemma}
    $\non^*_\omega(\ed)\leq \mathfrak{d}$.
\end{lemma}

With this, we are ready to calculate the value of $ \non^*_\omega(\ed)$.

\begin{theorem}
    $\cov(\mathcal{M}) = \non^*_\omega(\ed)$.
\end{theorem}

\begin{proof}

    Fix $\kappa < \cov(\mathcal{M})$ and $\{A_\alpha : \alpha < \kappa\} \subseteq \ed$. 
Define 
\[
F=\{ f \in \omega^\omega : \exists\, \alpha \ (A_\alpha = f)\}.
\]
By Lemma \ref{cov m}, there is $g \in \omega^\omega$ such that for every $f \in F$ there are infinitely many $n$ with $f(n)=g(n)$.
Let 
\[
X=\{g\}\cup\bigl\{\{n\}\times\omega : n<\omega\bigr\}
\]
witness $\kappa<\non^*_\omega(\ed)$.

Now let $\kappa<\non^*_\omega(\ed)$ and let $\{(x^\alpha,P^\alpha): \alpha<\kappa\}$ be a family of chopped reals.  
Since $\kappa<\mathfrak{d}$, by Lemma \ref{d-part}, there is $Q\in \mathrm{IP}$ such that for each $\alpha<\kappa$ there exist infinitely many $n$ and some $m$ with $P^\alpha_m\subseteq Q_n$.  

Let
\[
B_\alpha=\{n : \exists\, m \, (P^\alpha_m \subseteq Q_n)\},
\]
and let $\{b_n^\alpha : n<\omega\}$ be its increasing enumeration.  
Define $f_\alpha:\omega\to H(\omega)$ by
\[
f_\alpha(n)=(x^\alpha\!\upharpoonright Q_{b_0^\alpha},\, x^\alpha\!\upharpoonright Q_{b_1^\alpha},\, \dots ,\, x^\alpha\!\upharpoonright Q_{b_{\frac{n(n+1)}2}^\alpha}).
\]

Since $\kappa<\non^*_\omega(\ed)$, there exist functions $\{g_n:n<\omega\}\subseteq \omega^{H(\omega)}$ such that for every $\alpha<\kappa$ there is some $n$ with infinitely many $m$ satisfying $g_n(m)=f_\alpha(m)$.  
Without loss of generality, for all $m,n$ we may write
\[
g_n(m)=( s^{n,m}_0,\dots,s^{n,m}_{\frac{m(m+1)}2}),
\]
where each $s^{n,m}_i$ is a function, with $dom(s^{n,m}_i)=Q_{k_i}$ for some $k_i$, and $i<j$ implies $k_i<k_j$.

Construct recursively a sequence $\{y_i:i<\omega\}\subseteq \mathrm{Fn}(\omega,2)$ ($y_i$ is a function with the domain of $y_i$ a finite set of $\omega$)  such that $y_i\subseteq y_j$ whenever $i<j$, and such that there is a set $F\in[Q]^{\frac{i(i+1)}2}$ with $dom(y_i)=\bigcup F$.

Start with $y_0=g_0(0)$. Assume $y_n$ has been constructed and $dom(y_n)=\bigcup F$ for some finite $F\subseteq Q$. For each $i<n+1$, note that $g_i(n+1)$ has length $\frac{(n+1)(n+2)}2$, so there exists a $k_i$ with 
\[
Q_{k_i}\notin F\cup\{Q_{k_j}:j<i\}
\]
and $Q_{k_i}=dom(s^{i,n+1}_{j_i})$ for some $j_i<\frac{(n+1)(n+2)}2$.  
Define
\[
y_{n+1}=y_n\cup\{s^{i,n+1}_{j_i}:i<n+1\}.
\]
Note that $|y_{n+1}|=\frac{n(n+1)}2+n+1$.

Finally, choose $y\in 2^\omega$ extending $\bigcup_{n<\omega} y_n$.  
Fix $\alpha<\kappa$. By Lemma \ref{chopped}, it is suffices to show that $y\in Match(x^\alpha,P^\alpha)$.  
Choose $n$ such that for infinitely many $k$ we have $f_\alpha(k)=g_n(k)$.  
For each $m\ge n$, if $f_\alpha(m)=g_n(m)$ then there is $i_m$ such that  
$s^{n,m}_{i_m}\in g_n(m)$ and
\[
x^\alpha\!\upharpoonright Q_{b^\alpha_{i_m}}
 = s^{n,m}_{i_m}\!\upharpoonright Q_{b^\alpha_{i_m}}
 \subseteq y_m.
\]
Hence
\[
x^\alpha\!\upharpoonright P^\alpha_{j_m}
= y\!\upharpoonright P^\alpha_{j_m},
\]
for some $j_m$ with $P^\alpha_{j_m}\subseteq Q_{b^\alpha_{i_m}}$.
    \end{proof}

Finally, in \cite{poke} is proved that $\add^*_\omega(\fin\times \fin) =\mathfrak{b}$. This motivates our last result.

\begin{proposition}
    $\non^*(\fin\times \fin)=\non^*_\omega(\fin\times \fin) =\mathfrak{d}$
\end{proposition}
\begin{proof}
Take $\kappa<\non^*_\omega(\fin\times\fin)$ and $\{f_\alpha:\alpha<\kappa\}\subseteq \omega^\omega$. 
We need to show that this family is not dominating.  
Choose $\{g_n:n<\omega\}\subseteq \omega^\omega$ witnessing that $\{f_\alpha:\alpha<\kappa\}$ is not a witness for $\non^*_\omega(\fin\times\fin)$.  
Let $g$ dominate each $g_n$.  
Then, for every $\alpha<\kappa$, there are infinitely many $k$ such that $f_\alpha(k)\le g(k)$.
Thus $g$ is not dominated by any $f_\alpha$.

Now take $\kappa<\mathfrak{d}$.  
Since $\fin\times\fin$ is a tall ideal, it is enough to show that if $\mathcal{A}=\{A_\alpha:\alpha<\kappa\}\subseteq\fin\times\fin$, then there is a set $X$ such that for all $\alpha<\kappa$ we have $A_\alpha\cap X\neq^*\emptyset$.  
Let $\mathcal{A}$ be as above and define
\[
\mathcal{F}
=\bigl\{
f_\alpha:\alpha<\kappa \,\wedge\, 
A_\alpha\subseteq \{(n,m): m\le f_\alpha(n)\}
\bigr\}.
\]
Let $g$ be such that $g\nleq^* f_\alpha$ for all $\alpha<\kappa$.  
Define $X=\{(n,m): m\le g(n)\}$.
\end{proof}

\nocite{*}
\bibliographystyle{amsplain}

\bibliography{References}
\end{document}